\documentclass{amsart}
\usepackage{amsmath,amsthm,amsfonts,amssymb,color,graphicx,appendix,ulem,comment,hyperref,mathrsfs,nth,pifont}

\numberwithin{equation}{section}
\theoremstyle{plain}
\newtheorem{theorem}{\sc Theorem}[section]

\newtheorem{corollary}[theorem]{\sc Corollary}
\newtheorem{definition}[theorem]{\sc Definition}
\newtheorem{lemma}[theorem]{\sc Lemma}
\newtheorem{proposition}[theorem]{\sc Proposition}

\theoremstyle{remark}
\newtheorem{remark}[theorem]{\sc Remark}

\newcommand{\Ham}{\overline H}

\newcommand{\Con}{\operatorname{C}}
\newcommand{\UC}{\operatorname{UC}}

\newcommand{\Lip}{\operatorname{Lip}}
\newcommand{\Liploc}{\operatorname{Lip_{\mathrm{loc}}}}

\begin{document}

\title[Stochastic homogenization of a class of quasiconvex viscous HJ equations in 1D]{Stochastic homogenization of a class of quasiconvex viscous Hamilton-Jacobi equations in one space dimension}

\author[A.\ Yilmaz]{Atilla Yilmaz}
\address{Atilla Yilmaz, Department of Mathematics, Temple University, 1805 N.\ Broad Street, Philadelphia, PA 19122, USA}
\email{atilla.yilmaz@temple.edu}
\urladdr{http://math.temple.edu/$\sim$atilla/}

\date{October 4, 2020.}

\subjclass[2010]{35B27, 35F21, 60G10.} 
\keywords{Stochastic homogenization; viscous Hamilton-Jacobi equation; effective Hamiltonian; quasiconvexity; level-set convexity; viscosity solution; corrector; random potential; scaled hill condition.}

\begin{abstract}
	We prove homogenization for a class of viscous Hamilton-Jacobi equations in the stationary \& ergodic setting in one space dimension. Our assumptions include most notably the following:
	the Hamiltonian is of the form $G(p) + \beta V(x,\omega)$, the function $G$ is coercive and strictly quasiconvex, $\min G = 0$, $\beta>0$, the random potential $V$ takes values in $[0,1]$ with full support and it satisfies a hill condition that involves the diffusion coefficient. Our approach is based on showing that, for every direction outside of a bounded interval $(\theta_1(\beta),\theta_2(\beta))$, there is a unique sublinear corrector with certain properties. We obtain a formula for the effective Hamiltonian and deduce that it is coercive, identically equal to $\beta$ on $(\theta_1(\beta),\theta_2(\beta))$, and strictly monotone elsewhere.
\end{abstract}

\maketitle

\section{Introduction}\label{sec:intro}

Consider a Hamilton-Jacobi (HJ) equation of the form
\begin{equation}\label{eq:girishuzur}
\partial_tu^\epsilon = \epsilon\,\mathrm{tr}\left(A\left(\frac{x}{\epsilon},\omega\right)D_{xx}^2u^\epsilon\right) + H\left(D_xu^\epsilon,\frac{x}{\epsilon},\omega\right),\quad (t,x)\in(0,+\infty)\times\mathbb{R}^d,
\end{equation}
where $\omega$ is an element of a probability space $(\Omega,\mathcal{F},\mathbb{P})$, and $\epsilon>0$. Assume that the diffusion matrix $A(x,\omega)$ and the Hamiltonian $H(p,x,\omega)$ are stationary \& ergodic processes in $x$, the former is positive semidefinite (for all $x$) and the latter diverges (uniformly in $x$) as $|p|\to+\infty$. We refer to such HJ equations as inviscid if $A \equiv 0$ and viscous otherwise.

As $\epsilon\to0$, \eqref{eq:girishuzur} is said to homogenize to an inviscid HJ equation of the form
\begin{equation}\label{eq:firishuzur}
\partial_t\overline u = \Ham\left(D_x\overline u\right),\quad (t,x)\in(0,+\infty)\times\mathbb{R}^d,
\end{equation}
if, for any initial condition from a prescribed class, the unique viscosity solution of \eqref{eq:girishuzur} with that initial condition converges locally uniformly on $[0,+\infty)\times\mathbb{R}^d$ to the unique viscosity solution of \eqref{eq:firishuzur} with the same initial condition. The function $\Ham$ is called the effective Hamiltonian. 

\subsection{Brief overview of our results}\label{sub:overview}

In this paper, we study the homogenization of \eqref{eq:girishuzur} under the following additional assumptions: $d=1$, the diffusion coefficient (which replaces $A(x,\omega)$ and is denoted by $a(x,\omega)$) takes values in $(0,1]$, the Hamiltonian is separable, i.e., it is of the form
\begin{equation}\label{eq:septa}
	H(p,x,\omega) = G(p) + \beta V(x,\omega),
\end{equation}
$G$ is a nonnegative and strictly quasiconvex (a.k.a.\ level-set convex) function that vanishes at the origin, $V(\,\cdot\,,\omega)$ is a $[0,1]$-valued potential whose range includes $(0,1)$ for $\mathbb{P}$-a.e.\ $\omega$, and $\beta>0$. We also put various regularity conditions on $a(\,\cdot\,,\omega)$ and $V(\,\cdot\,,\omega)$, but we postpone such details to Section \ref{sec:results}. Last but not least, we impose what we call the scaled hill condition on the pair $(a,V)$ (see \eqref{eq:scaledhill}) which holds for wide and natural classes of examples, but fails (most notably) in the periodic case. See Appendix \ref{app:scaledhill} for details and references.

In the special case we described in the previous paragraph, we prove that, for $\mathbb{P}$-a.e.\ $\omega$, as $\epsilon\to0$, \eqref{eq:girishuzur} homogenizes to an inviscid HJ equation.
We establish this result first with linear initial data (see Theorem \ref{thm:homlin}) and then with uniformly continuous initial data (see Corollary \ref{cor:velinim}).
Moreover, we give a formula for the effective Hamiltonian and deduce the following:
$\Ham(\theta)$ is identically equal to $\beta$ on a bounded interval $(\theta_1(\beta),\theta_2(\beta))$ that contains $0$,
strictly decreasing on $(-\infty,\theta_1(\beta)]$,
strictly increasing on $[\theta_2(\beta),+\infty)$,
and divergent as $\theta\to\pm\infty$.

Our approach is based on correctors (see Subsection \ref{sub:literature} for a general and informal definition). We show that, for every $\theta\notin(\theta_1(\beta),\theta_2(\beta))$, there is a unique sublinear corrector in a certain class of functions (see Theorems \ref{thm:staticHJ}--\ref{thm:thetasOK} and Remark \ref{rem:nocor}) and the desired homogenization result with initial condition $x\mapsto \theta x$ follows. To cover $\theta\in(\theta_1(\beta),\theta_2(\beta))$ and obtain the flat piece of the graph of the effective Hamiltonian $\Ham$, we use the sublinear correctors for $\theta_1(\beta)$ and $\theta_2(\beta)$ in combination to construct subsolutions and supersolutions. This is where we rely on the scaled hill condition.

\subsection{Background and context}\label{sub:literature}

In the general setting of \eqref{eq:girishuzur}, given any $\theta\in\mathbb{R}^d$, if there exist $\lambda(\theta)\in\mathbb{R}$ and $F_\theta:\mathbb{R}^d\times\Omega\to\mathbb{R}$ such that 
\begin{equation}\label{eq:sakinhuzur}
\mathrm{tr}\left(A\left(x,\omega\right)D_{xx}^2F_\theta\right) + H\left(\theta + D_xF_\theta,x,\omega\right) = \lambda(\theta),\quad x\in\mathbb{R}^d,
\end{equation}
and $|F_\theta(x,\omega)| = o(|x|)$ as $|x|\to+\infty$ for $\mathbb{P}$-a.e.\ $\omega$, 
then $F_\theta$ is referred to as a sublinear corrector in the literature. The motivation behind this definition lies in the observation that
\[ u^\epsilon(t,x,\omega) = t\lambda(\theta) + \theta\cdot x +  \epsilon F_\theta\left(\frac{x}{\epsilon},\omega\right) \] 
gives a solution of \eqref{eq:girishuzur} and, for $\mathbb{P}$-a.e.\ $\omega$, as $\epsilon\to0$, it converges to $\overline u(t,x) = t\lambda(\theta) + \theta\cdot x$ which defines a solution of \eqref{eq:firishuzur} with $\Ham(\theta) = \lambda(\theta)$.

The first instances of sublinear correctors in the context of HJ equations were introduced in \cite{LPV87} when $d\ge1$, $A \equiv 0$ and $x = (x_1,\ldots,x_d)\mapsto H(p,x,\omega)$ is $1$-periodic in $x_i$ for each $i\in\{1,\ldots,d\}$.  The authors of that seminal paper used the compactness of the unit cube $[0,1]^d$ to prove that there is a periodic (and hence bounded) corrector for every $\theta\in\mathbb{R}^d$ and then provided some additional arguments to conclude that \eqref{eq:girishuzur} homogenizes to \eqref{eq:firishuzur} with $\Ham(\theta) = \lambda(\theta)$ as in the paragraph above. This result was subsequently adapted in \cite{E92} to the case where $d\ge1$, $A$ is positive definite and $(A,H)$ is periodic in the same way.

When $p\mapsto H(p,x,\omega)$ is convex, \eqref{eq:girishuzur} homogenizes for $\mathbb{P}$-a.e.\ $\omega$ without any periodicity assumption. Results of this form were first obtained in \cite{S99,RT00} for inviscid equations and then in \cite{LS05,KRV06} for their viscous counterparts.
The starting point of all four of the cited papers is a variational formula for the viscosity solutions of \eqref{eq:girishuzur} involving the convex conjugate of $p\mapsto H(p,x,\omega)$. The first three papers then apply the subadditive ergodic theorem to this variational formula whereas the fourth one uses ideas and techniques from the theory of large deviations (as outlined in \cite{K07}). In particular, none of them rely on the existence of sublinear correctors (although their connection to homogenization is given in \cite{RT00} as a separate result, cf.\ \cite{LS03}).

It is natural to ask if homogenization takes place under the weaker assumption of quasiconvexity, i.e., when the sublevel sets of $p\mapsto H(p,x,\omega)$ are convex. This question has been answered positively for inviscid equations in \cite{DS09} and \cite{AS13} when $d=1$ and $d\ge1$, respectively. The proof in \cite{DS09} involves correctors as well as approximate correctors which are solutions of \eqref{eq:sakinhuzur} when an error margin is introduced on the right-hand side of that equality, whereas the strategy in \cite{AS13} is to apply the subadditive ergodic theorem to certain solutions of \eqref{eq:sakinhuzur} when the condition $x\in\mathbb{R}^d$ there is replaced with $x\in\mathbb{R}^d\setminus\{y\}$ for any $y\in\mathbb{R}^d$, bypassing the existence of sublinear correctors.

To the best of our knowledge, outside of periodic settings, Theorem \ref{thm:homlin} and Corollary \ref{cor:velinim} are the first homogenization results for a class of viscous HJ equations with quasiconvex Hamiltonians that are not necessarily convex. The effective Hamiltonian that we provide in Theorem \ref{thm:homlin} has the same qualitative properties (recall the second paragraph of Subsection \ref{sub:overview}) as the effective Hamiltonian for the inviscid counterparts (covered by \cite{DS09,AS13}) of the equations we study.

If quasiconvexity is violated, then there is no general answer to the question of homogenization. Indeed, when $d\ge2$, there are positive results for certain classes of such HJ equations (see \cite{ATY15,CS17,AC18,QTY18,G19}) as well as negative results for others (see \cite{Z17,FeS17,FeFZ19+}). The counterexamples in the latter collection of papers involve Hamiltonians with saddle points, so they cannot be adapted to $d=1$. In fact, in one dimension, we expect \eqref{eq:girishuzur} to homogenize for $\mathbb{P}$-a.e.\ $\omega$ under only mild regularity and growth assumptions. This has already been proved in \cite{ATY16,G16} for inviscid equations. Moreover, in that case, if the original Hamiltonian is separable (as in \eqref{eq:septa}), then we have a detailed picture of the effective Hamiltonian (see \cite{Y20+}).

It has recently been shown that homogenization takes place at least for certain classes of viscous HJ equations in one dimension with Hamiltonians that are not quasiconvex but piecewise convex.
The first such result was given in \cite{DK17} which studies so-called pinned Hamiltonians, followed by \cite{YZ19,KYZ20,DK20+} which consider separable Hamiltonians that satisfy, in particular, hill and valley conditions that are closely related to our scaled hill condition (see Appendix \ref{app:scaledhill} for details). Since any continuous and coercive function on the real line is piecewise quasiconvex, we think that our results in this paper 
constitute a major step toward generalizing the aforementioned previous results and establishing homogenization for a wide class of viscous HJ equations in one dimension.

\section{Our results}\label{sec:results}

Throughout the paper, for any domain of the form $X = I\times J$ or $X = J$ where $I\subset[0,+\infty)$ and $J\subset\mathbb{R}$ are open intervals,
$\Con(X)$, $\UC(X)$, $\Lip(X)$ and $\Lip_{\mathrm{loc}}(X)$ stand for the space of continuous, uniformly continuous, Lipschitz continuous and locally Lipschitz continuous real-valued functions on $X$, respectively. 
Similarly, $\Con^k(X)$, $k\in\{1,2\}$, stand for the space of real-valued functions on $X$ with continuous derivatives of order $k$. These definitions extend to the closure of $X$ as usual.

Let $(\Omega,\mathcal{F},\mathbb{P})$ be a probability space equipped with a group of measure-preserving transformations $\tau_x:\Omega\to\Omega$, $x\in\mathbb{R}$, such that $(x,\omega)\mapsto \tau_x\omega$ is measurable. Assume that $\mathbb{P}$ is ergodic under this group of transformations, i.e.,
\[ \mathbb{P}(\cap_{x\in\mathbb{R}}\tau_xA) \in\{0,1\}\quad\text{for every}\ A\in\mathcal{F}. \]
Write $\mathbb{E}[\,\cdot\,]$ to denote expectation with respect to $\mathbb{P}$.

For every $\epsilon > 0$ and $\omega\in\Omega$, we consider the viscous HJ equation
\begin{equation}\label{eq:asilhuzur}
	\partial_tu^\epsilon = \epsilon a\left(\frac{x}{\epsilon},\omega\right)\partial_{xx}^2u^\epsilon + G(\partial_xu^\epsilon) + \beta V\left(\frac{x}{\epsilon},\omega\right),\quad (t,x)\in(0,+\infty)\times\mathbb{R},\tag{$\mathrm{HJ}_{\epsilon,\omega}$}
\end{equation}
under the following assumptions:
\begin{equation}\label{eq:Gcon}
\begin{aligned}
&\text{$G:\mathbb{R}\to [0,+\infty)$ is coercive, i.e.,}\ \lim_{p\to\pm\infty} G(p) = +\infty,\\
&\text{$G\in\Liploc(\mathbb{R})$;}\\
\end{aligned}
\end{equation}
\begin{equation}\label{eq:Gshape}
\begin{aligned}
	\text{$G(0) = 0$,}\ \  &\text{$G_1 := \left.G\right|_{(-\infty,0]}$ is strictly decreasing,}\\
	&\text{$G_2 := \left.G\right|_{[0,+\infty)}$ is strictly increasing;}
\end{aligned}
\end{equation}
\begin{equation}\label{eq:stat}
\begin{aligned}
	&\text{$a:\mathbb{R}\times\Omega\to(0,1]$ and $V:\mathbb{R}\times\Omega\to[0,1]$ are stationary, i.e.,}\\
	&\text{$a(x,\omega) = a(0,\tau_x\omega)$ and $V(x,\omega) = V(0,\tau_x\omega)$ for every $(x,\omega)\in\mathbb{R}\times\Omega$;}
\end{aligned}
\end{equation}
\begin{align}
	&\inf\{V(x,\omega):\,x\in\mathbb{R}\} = 0\ \text{and}\ \sup\{V(x,\omega):\,x\in\mathbb{R}\} = 1\ \text{for $\mathbb{P}$-a.e.\ $\omega$;}\label{eq:fruit}\\
	&\text{$a(\,\cdot\,,\omega)$ and $V(\,\cdot\,,\omega)$ are in $\Con(\mathbb{R})$ for every $\omega\in\Omega$;}\label{eq:aVcon1}
\end{align}
and $\beta > 0$  (which is fixed throughout the paper).

Two remarks are in order. First,
$G$ is strictly quasiconvex, i.e., 
\[ G(cp + (1-c)q) < \max\{G(p),G(q) \} \] whenever $p \ne q$ and $0 < c < 1$. (See Figure \ref{fig:G}.) Second, \eqref{eq:fruit} is essentially equivalent to
\[ \mathbb{P}( V(0,\omega) = h ) < 1\ \ \text{for every $h\in[0,1]$}. \]
This is due to ergodicity, the presence of the parameter $\beta$ and the observation that  adding a constant to the right-hand side of \eqref{eq:asilhuzur} corresponds to adding a linear (in time) term to $u^\epsilon$.

\begin{figure}
	\includegraphics{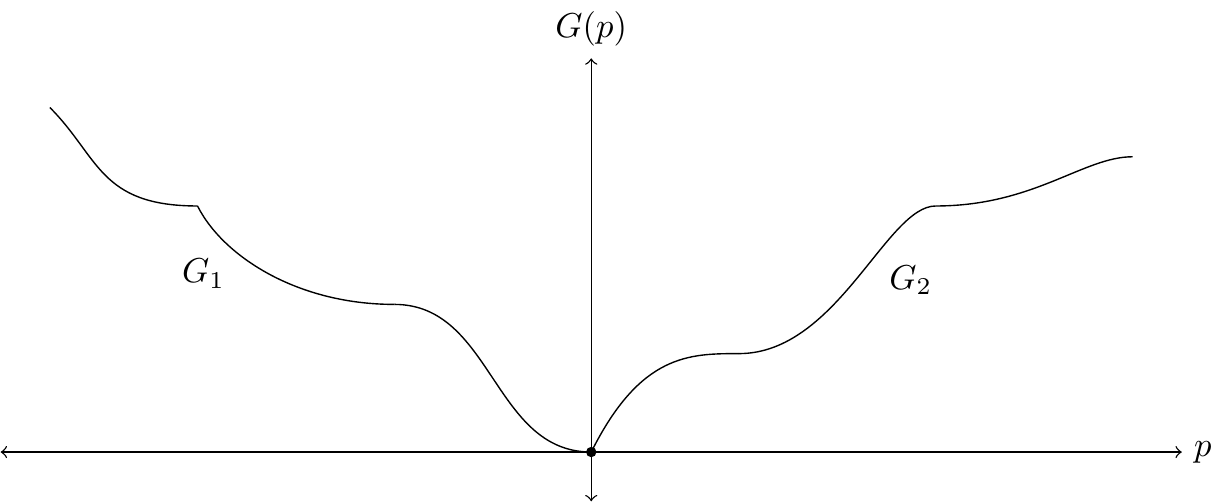}
	\caption{The graph of a function $G$ that satisfies \eqref{eq:Gcon}--\eqref{eq:Gshape}.}
	\label{fig:G}
\end{figure}

\subsection{The static HJ equation}

Our first couple of results are on the static (i.e., time-independent) version of \eqref{eq:asilhuzur} with $\epsilon = 1$. We prove them in Section \ref{sec:staticHJ}.

\begin{theorem}\label{thm:staticHJ}
	Assume \eqref{eq:Gcon}--\eqref{eq:aVcon1}. For every $\lambda\ge\beta$ and $\omega\in\Omega$, the static viscous HJ equation
	\begin{equation}\label{eq:palasvis}
		a(x,\omega)F'' + G(F') + \beta V(x,\omega) = \lambda,\quad x\in\mathbb{R},
	\end{equation} 
	has a unique solution $F_1^\lambda(\,\cdot\,,\omega)\in\Con^2(\mathbb{R})$ such that
	\[ F_1^\lambda(0,\omega) = 0 \quad\text{and}\quad (F_1^\lambda)'(x,\omega) \in[G_1^{-1}(\lambda),G_1^{-1}(\lambda - \beta)]\ \text{for all}\ x\in\mathbb{R}. \]
	Similarly, it has a unique solution $F_2^\lambda(\,\cdot\,,\omega)\in\Con^2(\mathbb{R})$ such that
	\begin{equation}\label{eq:adamvis2}
	F_2^\lambda(0,\omega) = 0 \quad\text{and}\quad (F_2^\lambda)'(x,\omega) \in[G_2^{-1}(\lambda - \beta),G_2^{-1}(\lambda)]\ \text{for all}\ x\in\mathbb{R}.
	\end{equation}
	Moreover, $(F_1^\lambda)'$ and $(F_2^\lambda)'$ are stationary, i.e.,
	\[ (F_i^\lambda)'(x,\omega) = (F_i^\lambda)'(0,\tau_x\omega) \]
	for every $i\in\{1,2\}$, $x\in\mathbb{R}$ and $\omega\in\Omega$.
\end{theorem}

\begin{theorem}\label{thm:thetasOK}
	Assume \eqref{eq:Gcon}--\eqref{eq:aVcon1}. With the notation in Theorem \ref{thm:staticHJ},
	\begin{equation}\label{eq:thetacom}
	\theta_1(\lambda) := \mathbb{E}[(F_1^\lambda)'(0,\omega)] \quad\text{and}\quad \theta_2(\lambda) := \mathbb{E}[(F_2^\lambda)'(0,\omega)]
	\end{equation}
	satisfy
	\[ \theta_1(\lambda)\in(G_1^{-1}(\lambda),G_1^{-1}(\lambda - \beta)) \quad\text{and}\quad \theta_2(\lambda) \in(G_2^{-1}(\lambda - \beta),G_2^{-1}(\lambda)) \]
	for every $\lambda\ge\beta$. 
	These quantities define two continuous bijections \[ \theta_1:[\beta,+\infty) \to (-\infty,\theta_1(\beta)] \quad\text{and}\quad \theta_2:[\beta,+\infty) \to [\theta_2(\beta),+\infty) \]
	which are decreasing and increasing, respectively. Moreover, their inverses $\theta_1^{-1}$ and $\theta_2^{-1}$ are locally Lipschitz continuous on their domains.
\end{theorem}

\begin{remark}\label{rem:nocor}
	Recall the first paragraph of Subsection \ref{sub:literature} and note that, for every $\lambda\ge\beta$ and $i\in\{1,2\}$, the function $(x,\omega) \mapsto F_i^\lambda(x,\omega) - \theta_i(\lambda)\cdot x$ is a sublinear corrector. 
	However, we will avoid the corrector terminology in the rest of the paper because we will work directly with $F_1^\lambda$ and $F_2^\lambda$ rather than their sublinearized versions.
\end{remark}

\subsection{Homogenization}

When $\epsilon = 1$, we drop the superscript of $u^\epsilon$ in \eqref{eq:asilhuzur} and write
\begin{equation}\label{eq:birhuzur}
\partial_tu = a(x,\omega)\partial_{xx}^2u + G(\partial_xu) + \beta V(x,\omega),\quad (t,x)\in(0,+\infty)\times\mathbb{R}.\tag{$\mathrm{HJ}_{\omega}$}
\end{equation}
We assume that,
\begin{equation}\label{eq:existunique}
\begin{aligned}
&\text{for every $\omega\in\Omega$ and $\theta\in\mathbb{R}$, \eqref{eq:birhuzur} has a unique viscosity solution}\\
&\text{$u_\theta(\,\cdot\,,\,\cdot\,,\omega)\in\UC([0,+\infty)\times\mathbb{R})$ such that $u_\theta(0,x,\omega) = \theta x$ for all $x\in\mathbb{R}$,}
\end{aligned}
\end{equation}
which carries over to \eqref{eq:asilhuzur} with an arbitrary $\epsilon > 0$.
Indeed, the unique viscosity solution of the latter equation with the same initial condition is given by
\[ u_\theta^\epsilon(t,x,\omega) = \epsilon u_\theta\left(\frac{t}{\epsilon},\frac{x}{\epsilon},\omega\right). \]
See Subsection \ref{sub:viscosity} for some preliminaries regarding viscosity solutions.

We strengthen assumption \eqref{eq:aVcon1} as follows:
\begin{equation}
\text{$\sqrt{a(\,\cdot\,,\omega)}\in\Lip(\mathbb{R})$ and $V(\,\cdot\,,\omega)\in\UC(\mathbb{R})$ for every $\omega\in\Omega$.}\label{eq:aVcon2}
\end{equation}
In Subsection \ref{sub:nonflat}, we use Theorems \ref{thm:staticHJ}--\ref{thm:thetasOK} and a comparison principle to prove that, for each $\theta\notin(\theta_1(\beta),\theta_2(\beta))$, the function $u_\theta^\epsilon(\,\cdot\,,\,\cdot\,,\omega)$ converges locally uniformly as $\epsilon\to0$ for $\mathbb{P}$-a.e.\ $\omega$.
Then, in Subsection \ref{sub:flat}, we obtain the same result for each $\theta\in(\theta_1(\beta),\theta_2(\beta))$ under the following additional assumption:
\begin{equation}\label{eq:scaledhill}
\begin{aligned}
	&\text{for every $h\in(0,1)$, $C>0$ and $\mathbb{P}$-a.e.\ $\omega$, there is an interval $[L_1,L_2]$ such that}\\
	&\int_{L_1}^{L_2}\frac{dy}{a(y,\omega)} \ge C \ \text{and}\ V(\,\cdot\,,\omega) \ge h\ \text{on}\ [L_1,L_2].
\end{aligned}
\end{equation}
We refer to \eqref{eq:scaledhill} as the scaled hill condition. See Appendix \ref{app:scaledhill} for a detailed discussion.

The set of $\omega$ for which $u_\theta^\epsilon(\,\cdot\,,\,\cdot\,,\omega)$ does not converge locally uniformly as $\epsilon\to0$ is a $\mathbb{P}$-null set, but it may depend on $\theta$. In order to treat all $\theta\in\mathbb{R}$ simultaneously, we make the following assumption:
\begin{equation}\label{eq:duzlip}
\begin{aligned}
	 &\text{for every $\omega\in\Omega$ and $\theta\in\mathbb{R}$, there exists an $\ell_\theta = \ell_\theta(\omega) > 0$ such that}\\
	 &|u_\theta(t,x,\omega) - u_\theta(t,y,\omega)| \le \ell_\theta|x-y|\ \ \text{for all}\ t\in[0,+\infty)\ \text{and}\ x,y\in\mathbb{R}.
\end{aligned}
\end{equation}

Here is the precise statement of our homogenization result with linear initial data. We complete its proof in Subsection \ref{sub:homproofs}.

\begin{theorem}\label{thm:homlin}
	Assume \eqref{eq:Gcon}--\eqref{eq:fruit} and \eqref{eq:existunique}--\eqref{eq:duzlip}. Define $\Ham\in\Liploc(\mathbb{R})$ by	
	\[
	   \Ham(\theta) = \begin{cases} \theta_1^{-1}(\theta) &\text{for}\ \theta\in(-\infty,\theta_1(\beta)]\quad (\text{strictly decreasing}),\\
													\beta &\text{for}\ \theta\in(\theta_1(\beta),\theta_2(\beta))\quad\text{(flat piece)},\\
									\theta_2^{-1}(\theta) &\text{for}\ \theta\in[\theta_2(\beta),+\infty)\quad (\text{strictly increasing}),
	\end{cases}
	\]
	with the continuous bijections $\theta_1$ and $\theta_2$ in Theorem \ref{thm:thetasOK}. For $\mathbb{P}$-a.e.\ $\omega$, as $\epsilon\to0$, when subject to linear initial data, \eqref{eq:asilhuzur} homogenizes to the inviscid HJ equation
	\begin{equation}\label{eq:efhuzur}
		\partial_t\overline u = \Ham(\partial_x\overline u),\quad(t,x)\in(0,+\infty)\times\mathbb{R}.\tag{$\overline{\mathrm{HJ}}$}
	\end{equation}
	Precisely, there exists an $\Omega_0\in\mathcal{F}$ with $\mathbb{P}(\Omega_0) = 1$ such that, for every $\omega\in\Omega_0$ and $\theta\in\mathbb{R}$, as $\epsilon\to0$, the unique viscosity solution $u_\theta^\epsilon(\,\cdot\,,\,\cdot\,,\omega)$ of \eqref{eq:asilhuzur} with the initial condition $u_\theta^\epsilon(0,x,\omega) = \theta x$, $x\in\mathbb{R}$, converges locally uniformly on $[0,+\infty)\times\mathbb{R}$ to $\overline u_\theta$ defined by
	\[ \overline u_\theta(t,x) = t\Ham(\theta) + \theta x, \]
	which is the unique (classical and hence viscosity) solution of \eqref{eq:efhuzur} with the same initial condition. 
\end{theorem}

Finally, replacing \eqref{eq:existunique} with the stronger assumption that
\begin{equation}\label{eq:well}
\begin{aligned}
	&\text{the Cauchy problem for \eqref{eq:birhuzur} is well-posed in $\UC([0,+\infty)\times\mathbb{R})$ for every $\omega\in\Omega$,}
\end{aligned}
\end{equation}
we generalize Theorem \ref{thm:homlin} to uniformly continuous initial data by citing a result from \cite{DK17} which is based on the perturbed test function method (see \cite{E89}).

\begin{corollary}\label{cor:velinim}
	 Assume \eqref{eq:Gcon}--\eqref{eq:fruit} and \eqref{eq:aVcon2}--\eqref{eq:well}. For $\mathbb{P}$-a.e.\ $\omega$, as $\epsilon\to0$, when subject to uniformly continuous initial data, \eqref{eq:asilhuzur} homogenizes to the inviscid HJ equation \eqref{eq:efhuzur} with the effective Hamiltonian $\Ham$ in Theorem \ref{thm:homlin}. Precisely, there exists an $\Omega_0\in\mathcal{F}$ with $\mathbb{P}(\Omega_0) = 1$ such that, for every $\omega\in\Omega_0$ and $g\in\UC(\mathbb{R})$, as $\epsilon\to0$, the unique viscosity solution $u_g^\epsilon(\,\cdot\,,\,\cdot\,,\omega)$ of \eqref{eq:asilhuzur} with the initial condition $u_g^\epsilon(0,\,\cdot\,,\omega) = g(\,\cdot\,)$ converges locally uniformly on $[0,+\infty)\times\mathbb{R}$ to the unique viscosity solution $\overline u_g$ of \eqref{eq:efhuzur} with the same initial condition.
\end{corollary}

See Theorem \ref{thm:suffcon} for a set of natural conditions (which are not meant to be sharp) under which \eqref{eq:duzlip}--\eqref{eq:well} are valid.

\section{The static HJ equation}\label{sec:staticHJ}

In the following two lemmas leading to the proof of Theorem \ref{thm:staticHJ}, we drop $\omega$ and assume that
\begin{equation}\label{eq:zengin}
	\text{$a:\mathbb{R}\to(0,1]$ and $V:\mathbb{R}\to[0,1]$ are in $\Con(\mathbb{R})$.}
\end{equation}

\begin{lemma}\label{lem:kirbiryy}
	Assume \eqref{eq:Gcon}--\eqref{eq:Gshape} and \eqref{eq:zengin}. For every $\lambda\ge\beta$, $L\in\mathbb{R}$ and $c\in[G_2^{-1}(\lambda - \beta),G_2^{-1}(\lambda)]$,
	the equation
	\begin{equation}\label{eq:palasLL}
		a(x)f'(x) + G(f(x)) + \beta V(x) = \lambda,\quad x\in[L,+\infty),
	\end{equation} 
	has a unique (classical) solution $f_2^\lambda(\,\cdot\,|\,L,c)$ that satisfies
	\begin{equation}\label{eq:topcer}
		f_2^\lambda(L\,|\,L,c) = c.
	\end{equation}
	Moreover,
	\begin{equation}\label{eq:adamcer}
		f_2^\lambda(x\,|\,L,c) \in[G_2^{-1}(\lambda - \beta),G_2^{-1}(\lambda)]\ \ \text{for all}\ x\in[L,+\infty).
	\end{equation}
\end{lemma}

\begin{proof}
	We rearrange \eqref{eq:palasLL} and write
	\[ f'(x) = \frac1{a(x)}\left(\lambda - G(f(x)) - \beta V(x)\right). \]
	By the Picard-Lindel\"of theorem, there is a unique local solution $f_2^\lambda(\,\cdot\,|\,L,c)$ in a neighborhood of $L$ that satisfies \eqref{eq:topcer}. Let us check that there is no blow-up at any $x\in(L,+\infty)$.
	\begin{itemize}
		\item If $f_2^\lambda(x\,|\,L,c) < G_2^{-1}(\lambda - \beta)$ for some $x \in(L,+\infty)$, then
		\[ \lambda - G(f_2^\lambda(x\,|\,L,c)) - \beta V(x) > \lambda - (\lambda - \beta) - \beta = 0 \]
		and $(f_2^\lambda)'(x\,|\,L,c) > 0$. We deduce that $f_2^\lambda(x\,|\,L,c) \ge G_2^{-1}(\lambda - \beta)$ for all $x\in[L,+\infty)$.
		
		\item If $f_2^\lambda(x\,|\,L,c) > G_2^{-1}(\lambda)$ for some $x \in(L,+\infty)$, then 
		\[ \lambda - G(f_2^\lambda(x\,|\,L,c)) - \beta V(x) < \lambda - \lambda - 0 = 0 \]
		and $(f_2^\lambda)'(x\,|\,L,c) < 0$. We deduce that $f_2^\lambda(x\,|\,L,c) \le G_2^{-1}(\lambda)$ for all $x\in[L,+\infty)$.	
	\end{itemize}
	We conclude that $f_2^\lambda(\,\cdot\,|\,L,c)$ is the unique solution of \eqref{eq:palasLL} that satisfies \eqref{eq:topcer}. Moreover, the bounds in \eqref{eq:adamcer} hold.
\end{proof}

\begin{lemma}\label{lem:kiriki}
	Assume \eqref{eq:Gcon}--\eqref{eq:Gshape} and \eqref{eq:zengin}. For every $\lambda\ge\beta$, the equation
	\begin{equation}\label{eq:palasBB}
		a(x)f'(x) + G(f(x)) + \beta V(x) = \lambda,\quad x\in\mathbb{R},
	\end{equation} 
	has a unique solution $f_2^\lambda\in\Con^1(\mathbb{R})$ such that
	\begin{equation}\label{eq:adamBB}
		f_2^\lambda(x) \in [G_2^{-1}(\lambda - \beta),G_2^{-1}(\lambda)]\ \ \text{for all}\ x\in\mathbb{R}.
	\end{equation}
\end{lemma}

\begin{proof}
	For every $\lambda \ge \beta$, there is a strictly increasing and continuous function $m_2^\lambda$ 
	such that
	\begin{equation}\label{eq:GLB}
	 	G_2(p + q) - G_2(p) \ge m_2^\lambda(q)\ \ \text{whenever}\ \ G_2^{-1}(\lambda - \beta) \le p \le p + q \le G_2^{-1}(\lambda).
	\end{equation}
	Equivalently, $(m_2^\lambda)^{-1}$ is a modulus of continuity for $G_2^{-1}$ on $[\lambda - \beta,\lambda]$. Without loss of generality, assume that
	\[ m_2^\lambda(q) \le q. \]
	
	Fix any $L\in\mathbb{R}$ and $c,d\in [G_2^{-1}(\lambda - \beta), G_2^{-1}(\lambda)]$ such that $c < d$. Recall Lemma \ref{lem:kirbiryy} and let $f_2^\lambda(\,\cdot\,|\,L,c)$ and $f_2^\lambda(\,\cdot\,|\,L,d)$ be the unique solutions of \eqref{eq:palasLL} that satisfy
	\[ f_2^\lambda(L\,|\,L,c) = c \quad\text{and}\quad f_2^\lambda(L\,|\,L,d) = d. \]
	Note that
	\begin{align*}
		\lambda &= a(x)(f_2^\lambda)'(x\,|\,L,c) + G_2(f_2^\lambda(x\,|\,L,c)) + \beta V(x)\\
				&= a(x)(f_2^\lambda)'(x\,|\,L,d) + G_2(f_2^\lambda(x\,|\,L,d)) + \beta V(x)
	\end{align*}
	for every $x\in[L,+\infty)$ by \eqref{eq:adamcer}. Rearranging the second equality, we get
	\[ a(x) \left[ (f_2^\lambda)'(x\,|\,L,d) - (f_2^\lambda)'(x\,|\,L,c) \right] + G_2(f_2^\lambda(x\,|\,L,d)) - G_2(f_2^\lambda(x\,|\,L,c)) = 0. \]
	It follows that $f_2^\lambda(x\,|\,L,d) \ge f_2^\lambda(x\,|\,L,c)$ for every $x\in[L,+\infty)$. Therefore,
	\[ h_2^\lambda(x\,|\,L,c,d) := f_2^\lambda(x\,|\,L,d) - f_2^\lambda(x\,|\,L,c) \]
	is a nonnegative function in $\Con^1([L,+\infty))$ such that
	\[ h_2^\lambda(L\,|\,L,c,d) = d - c \in (0,G_2^{-1}(\lambda)] \]
	and
	\[ a(x)(h_2^\lambda)'(x\,|\,L,c,d) + m_2^\lambda(h_2^\lambda(x\,|\,L,c,d)) \le 0 \]
	by \eqref{eq:GLB}. We apply a variant of the Gr\"onwall-Bellman lemma (see Appendix \ref{app:gronwall}) and deduce that
	\begin{equation}\label{eq:quantbdiki}
	h_2^\lambda(x\,|\,L,c,d) \le (\Phi_2^\lambda)^{-1}\left(\int_L^x\frac{dy}{a(y)}\right),
	\end{equation}
	where
	\[ \Phi_2^\lambda(p) = \int_p^{G_2^{-1}(\lambda)}\frac{dq}{m_2^\lambda(q)} \] 
	and
	\begin{equation}\label{eq:iyiya}
		\lim_{z\to+\infty} (\Phi_2^\lambda)^{-1}(z) = 0.
	\end{equation}

	Fix any $x\in\mathbb{R}$. For every $L\in(-\infty,x)$, $c,d\in[G_2^{-1}(\lambda-\beta),G_2^{-1}(\lambda)]$ and $L',L''\in (-\infty,L]$, we can restrict the functions $f_2^\lambda(\,\cdot\,|\,L',c)$ and $f_2^\lambda(\,\cdot\,|\,L'',d)$ to the interval $[L,+\infty)$. If
	\[ f_2^\lambda(L\,|\,L',c) = f_2^\lambda(L\,|\,L'',d), \]
	then $f_2^\lambda(x\,|\,L',c) = f_2^\lambda(x\,|\,L'',d)$ by the uniqueness in Lemma \ref{lem:kirbiryy}. Otherwise, we use \eqref{eq:quantbdiki} to deduce that
	\[ |f_2^\lambda(x\,|\,L',c) - f_2^\lambda(x\,|\,L'',d)| \le (\Phi_2^\lambda)^{-1}\left(\int_L^x\frac{dy}{a(y)}\right). \]
	By \eqref{eq:iyiya} and the Cauchy criterion for convergence, the limit
	\[ f_2^\lambda(x) := \lim_{L\to -\infty}f_2^\lambda(x\,|\,L,c) \]
	exists, $f_2^\lambda(x) \in [G_2^{-1}(\lambda - \beta),G_2^{-1}(\lambda)]$ and it is independent of $c\in [G_2^{-1}(\lambda - \beta),G_2^{-1}(\lambda)]$.
	
	Fix any $c \in [G_2^{-1}(\lambda - \beta),G_2^{-1}(\lambda)]$. Note that, for every $x_1,x_2\in\mathbb{R}$ such that $x_1 < x_2$,
	\begin{align*}
	g_2^\lambda(x) &:= \lim_{L\to-\infty}(f_2^\lambda)'(x\,|\,L,c) = \lim_{L\to-\infty}\frac1{a(x)}\left(\lambda - G(f_2^\lambda(x\,|\,L,c)) - \beta V(x)\right)\\
	&\; = \frac1{a(x)}\left(\lambda - G(f_2^\lambda(x)) - \beta V(x)\right)
	\end{align*}
	and the limits are uniform on $[x_1,x_2]$. Consequently,
	\begin{align*}
	\int_{x_1}^{x_2}g_2^\lambda(x)dx &= \lim_{L\to-\infty} \int_{x_1}^{x_2}(f_2^\lambda)'(x\,|\,L,c)dx = \lim_{L\to-\infty}( f_2^\lambda(x_2\,|\,L,c) - f_2^\lambda(x_1\,|\,L,c) )\\
	&= f_2^\lambda(x_2) - f_2^\lambda(x_1)
	\end{align*}
	and $g_2^\lambda(x) = (f_2^\lambda)'(x)$. We conclude that $f_2^\lambda$ is a solution of \eqref{eq:palasBB} in $\Con^1(\mathbb{R})$.
	
	Finally, once we impose the bounds in \eqref{eq:adamBB}, uniqueness follows. Indeed, suppose $f_2^\lambda,\tilde f_2^\lambda\in\Con^1(\mathbb{R})$ are solutions of \eqref{eq:palasBB} that both satisfy \eqref{eq:adamBB}. For any $x\in\mathbb{R}$, if there exists an $L\in (-\infty,x)$ such that $f_2^\lambda(L) = \tilde f_2^\lambda(L)$, then $f_2^\lambda(x) = \tilde f_2^\lambda(x)$ by the uniqueness in Lemma \ref{lem:kirbiryy}. Otherwise,
	\[ 0 \le |f_2^\lambda(x) - \tilde f_2^\lambda(x)| \le \lim_{L\to-\infty}(\Phi_2^\lambda)^{-1}\left(\int_L^x\frac{dy}{a(y)}\right) = 0 \]
	as in \eqref{eq:quantbdiki}--\eqref{eq:iyiya}.
\end{proof}

We are ready to go back to the stochastic setting.

\begin{proof}[Proof of Theorem \ref{thm:staticHJ}]
	For every $\lambda\ge\beta$ and $\omega\in\Omega$, by Lemma \ref{lem:kiriki}, the equation
    \[ a(x,\omega)f' + G(f) + \beta V(x,\omega) = \lambda,\quad x\in\mathbb{R}, \]
    has a unique solution $f_2^\lambda(\,\cdot\,,\omega)\in\Con^1(\mathbb{R})$ such that $f_2^\lambda(x,\omega) \in[G_2^{-1}(\lambda - \beta),G_2^{-1}(\lambda)]$ for all $x\in\mathbb{R}$. Define $F_2^\lambda(\,\cdot\,,\omega)$ by setting
    \[ F_2^\lambda(x,\omega) = \int_0^x f_2^\lambda(y,\omega)dy \]
    for every $x\in\mathbb{R}$. It follows immediately that $F_2^\lambda(\,\cdot\,,\omega)$ is the unique solution of \eqref{eq:palasvis} in $\Con^2(\mathbb{R})$ that satisfies \eqref{eq:adamvis2}.
    The uniqueness of the solution in this class, in combination with the stationarity of the functions $a$ and $V$, implies that $(F_2^\lambda)'$ is stationary. Indeed, for every $x,y\in\mathbb{R}$ and $\omega\in\Omega$,
    \begin{align*}
    	\lambda &= a(x+y,\omega)(F_2^\lambda)''(x+y,\omega) + G((F_2^\lambda)'(x+y,\omega)) + \beta V(x+y,\omega)\\
    	&= a(x,\tau_y\omega)(F_2^\lambda)''(x+y,\omega) + G((F_2^\lambda)'(x+y,\omega)) + \beta V(x,\tau_y\omega).
    \end{align*}
    Therefore, $F_2^\lambda(x+y,\omega) - F_2^\lambda(y,\omega) = F_2^\lambda(x,\tau_y\omega)$ and $(F_2^\lambda)'(x+y,\omega) = (F_2^\lambda)'(x,\tau_y\omega)$. 
    
    For every $p,x\in\mathbb{R}$ and $\omega\in\Omega$, let
    \begin{equation}\label{eq:lazimol}
	    \tilde G(p) = G(-p),\quad \tilde a(x,\omega) = a(-x,\omega),\quad \tilde V(x,\omega) = V(-x,\omega) \quad\text{and}\quad \tilde F(x,\omega) = F(-x,\omega).
	\end{equation}
    After these substitutions, \eqref{eq:palasvis} becomes
    \[ \tilde a(x,\omega)\tilde F'' + \tilde G(\tilde F') + \beta \tilde V(x,\omega) = \lambda,\quad x\in\mathbb{R}. \]
    Moreover, assumptions \eqref{eq:Gcon}--\eqref{eq:aVcon1} translate to $\tilde G$, $\tilde a$ and $\tilde V$ (if we introduce and work with $\tilde\tau_x = \tau_{-x}$).
    The desired conclusions regarding the existence \& uniqueness of $F_1^\lambda$ and the stationarity of $(F_1^\lambda)'$ follow.
\end{proof}

For every $\lambda\ge\beta$ and $\mathbb{P}$-a.e.\ $\omega$,
\begin{equation}\label{eq:ardil}
\begin{aligned}
	  &\lim_{x\to\pm\infty}\frac1{x}F_2^\lambda(x,\omega) = \lim_{x\to\pm\infty}\frac1{x} \left(F_2^\lambda(x,\omega) - F_2^\lambda(0,\omega)\right)\\
	= &\lim_{x\to\pm\infty}\frac1{x}\int_0^x (F_2^\lambda)'(y,\omega)dy = \mathbb{E}[(F_2^\lambda)'(0,\omega)] = \theta_2(\lambda)
\end{aligned}
\end{equation}
by Theorem \ref{thm:staticHJ}, the Birkhoff ergodic theorem and the definition of $\theta_2(\lambda)$ in \eqref{eq:thetacom}.

\begin{lemma}\label{lem:sertbas}
	Assume \eqref{eq:Gcon}--\eqref{eq:aVcon1}. For every $\lambda\ge\beta$ and $\epsilon \in (0,1)$,
	\[ \theta_2(\lambda + \epsilon) - \theta_2(\lambda) \ge \epsilon/{\tilde\kappa_2^\lambda}, \]
	where $\tilde\kappa_2^\lambda$ is a Lipschitz constant for $G_2$ on $[G_2^{-1}(\lambda - \beta),G_2^{-1}(\lambda + 1)]$.
\end{lemma}

\begin{proof}
	For every $\lambda\ge\beta$, $\epsilon\in(0,1)$ and $\epsilon'\in(0,\epsilon)$, let $\delta = (\epsilon - \epsilon')/{\tilde\kappa_2^\lambda}$ and note that
	\begin{align*}
		&a(x,\omega)(F_2^\lambda)''(x,\omega) + G_2((F_2^\lambda)'(x,\omega) + \delta) + \beta V(x,\omega)\\
	\le\ &a(x,\omega)(F_2^\lambda)''(x,\omega) + G_2((F_2^\lambda)'(x,\omega)) + \beta V(x,\omega) + \epsilon - \epsilon' = \lambda + \epsilon - \epsilon',\quad (x,\omega)\in\mathbb{R}\times\Omega,
	\end{align*}
	by \eqref{eq:palasvis} and the following bounds due to \eqref{eq:adamvis2}: 
	\[ \lambda - \beta \le G_2((F_2^\lambda)'(x,\omega)) < G_2((F_2^\lambda)'(x,\omega) + \delta) < G_2((F_2^\lambda)'(x,\omega) + 1/{\tilde\kappa_2^\lambda}) \le \lambda + 1. \]
	Since
	\begin{equation}\label{eq:vpkonus}
	a(x,\omega)(F_2^{\lambda + \epsilon})''(x,\omega) + G_2((F_2^{\lambda + \epsilon})'(x,\omega)) + \beta V(x,\omega) = \lambda + \epsilon,\quad (x,\omega)\in\mathbb{R}\times\Omega,
	\end{equation}
	we deduce that
	\begin{equation}\label{eq:baybu}
		a(x,\omega)\left[(F_2^{\lambda + \epsilon})''(x,\omega) - (F_2^\lambda)''(x,\omega)\right] + G_2((F_2^{\lambda + \epsilon})'(x,\omega)) - G_2((F_2^\lambda)'(x,\omega) + \delta) \ge \epsilon'
	\end{equation}
	for every $x\in\mathbb{R}$ and $\omega\in\Omega$.
	
	Let
	\begin{equation}\label{eq:huso}
		h_2^{\lambda,\epsilon}(x,\omega) = (F_2^{\lambda + \epsilon})'(x,\omega) - (F_2^\lambda)'(x,\omega)
	\end{equation}
	and note that $h_2^{\lambda,\epsilon}(x,\omega) \ge -G_2^{-1}(\lambda)$ by \eqref{eq:adamvis2}. If $-G_2^{-1}(\lambda) \le h_2^{\lambda,\epsilon}(x,\omega) \le \delta$, then
	\[ (h_2^{\lambda,\epsilon})'(x,\omega) \ge a(x,\omega)(h_2^{\lambda,\epsilon})'(x,\omega) \ge \epsilon' \]
	by \eqref{eq:baybu}. Hence, for every $x_1\in\mathbb{R}$, there exists an $x_2\in [x_1,x_1 + \frac{G_2^{-1}(\lambda) + \delta}{\epsilon'}]$ such that $h_2^{\lambda,\epsilon}(x,\omega) \ge \delta$ whenever $x\ge x_2$. Therefore, in fact, $h_2^{\lambda,\epsilon}(x,\omega) \ge \delta$ for every $x\in\mathbb{R}$. We recall \eqref{eq:ardil} and conclude that
	\[ \theta_2(\lambda + \epsilon) - \theta_2(\lambda) = \lim_{x\to+\infty}\frac1{x}\int_0^x h_2^{\lambda,\epsilon}(y,\omega)dy \ge \delta = (\epsilon - \epsilon')/{\tilde\kappa_2^\lambda}. \]
	Since $\epsilon'\in(0,\epsilon)$ is arbitrary, the desired inequality follows.
\end{proof}

\begin{lemma}\label{lem:tarkan}
	Assume \eqref{eq:Gcon}--\eqref{eq:aVcon1}. For every $\lambda\ge\beta$ and $\epsilon \in (0,1)$,
	\[ \theta_2(\lambda + \epsilon) - \theta_2(\lambda) \le (\tilde m_2^\lambda)^{-1}(\epsilon), \]
	where $(\tilde m_2^\lambda)^{-1}$ is a modulus of continuity for $G_2^{-1}$ on $[\lambda - \beta,\lambda + 1]$.
\end{lemma}

\begin{proof}
	For every $\lambda\ge\beta$, $\epsilon\in(0,1)$ and $\epsilon'\in(0,1-\epsilon)$, let $\delta = (\tilde m_2^\lambda)^{-1}(\epsilon + \epsilon')$ and note that
	\begin{align*}
	& a(x,\omega)(F_2^\lambda)''(x,\omega) + G_2((F_2^\lambda)'(x,\omega) + \delta) + \beta V(x,\omega)\\
	\ge\ &a(x,\omega)(F_2^\lambda)''(x,\omega) + G_2((F_2^\lambda)'(x,\omega)) + \beta V(x,\omega) + \epsilon + \epsilon' = \lambda + \epsilon + \epsilon',\quad (x,\omega)\in\mathbb{R}\times\Omega,
	\end{align*}
	by \eqref{eq:palasvis} and the following bounds due to \eqref{eq:adamvis2}:
	\[ \lambda - \beta \le G_2((F_2^\lambda)'(x,\omega)) < G_2((F_2^\lambda)'(x,\omega)) + \epsilon + \epsilon' < G_2((F_2^\lambda)'(x,\omega)) + 1 \le \lambda + 1. \]
	Comparing this inequality with $\eqref{eq:vpkonus}$, we deduce that
	\begin{equation}\label{eq:crank}
		a(x,\omega)\left[ (F_2^{\lambda + \epsilon})''(x,\omega) - (F_2^\lambda)''(x,\omega) \right] + G_2((F_2^{\lambda + \epsilon})'(x,\omega)) - G_2((F_2^\lambda)'(x,\omega) + \delta) \le -\epsilon'
	\end{equation}
	for every $x\in\mathbb{R}$ and $\omega\in\Omega$.
	
	Recall
	\[ h_2^{\lambda,\epsilon}(x,\omega) = (F_2^{\lambda + \epsilon})'(x,\omega) - (F_2^\lambda)'(x,\omega) \]
	from \eqref{eq:huso} and note that $h_2^{\lambda,\epsilon}(x,\omega) \le G_2^{-1}(\lambda + 1)$ by \eqref{eq:adamvis2}. If $\delta \le h_2^{\lambda,\epsilon}(x,\omega) \le G_2^{-1}(\lambda + 1)$, then
	\[ (h_2^{\lambda,\epsilon})'(x,\omega) \le a(x,\omega)(h_2^{\lambda,\epsilon})'(x,\omega) \le -\epsilon' \]
	by \eqref{eq:crank}. Hence, for every $x_1\in\mathbb{R}$, there exists an $x_2\in [x_1,x_1 + \frac{G_2^{-1}(\lambda + 1) - \delta}{\epsilon'}]$ such that $h_2^{\lambda,\epsilon}(x,\omega) \le \delta$ whenever $x\ge x_2$. Therefore, in fact, $h_2^{\lambda,\epsilon}(x,\omega) \le \delta$ for every $x\in\mathbb{R}$. We recall \eqref{eq:ardil} and conclude that	
	\[ \theta_2(\lambda + \epsilon) - \theta_2(\lambda) = \lim_{x\to+\infty}\frac1{x}\int_0^x h_2^{\lambda,\epsilon}(y,\omega)dy \le \delta = (\tilde m_2^\lambda)^{-1}(\epsilon + \epsilon').\]
	Since $\epsilon'\in(0,1 - \epsilon)$ is arbitrary, the desired inequality follows.
\end{proof}

\begin{proof}[Proof of Theorem \ref{thm:thetasOK}]
	For every $\lambda\ge\beta$ and $\omega\in\Omega$,
	\[ (F_2^\lambda)''(x,\omega) = \frac1{a(x,\omega)}(\lambda - G((F_2^\lambda)'(x,\omega)) - \beta V(x,\omega)),\quad x\in\mathbb{R}. \]
	Since $\mathbb{P}(V(0,\omega) \in (0,1)) > 0$, it follows (by a slight modification of the itemized argument in the proof of Lemma \ref{lem:kirbiryy}) that
	\[ \mathbb{P}\left( (F_2^\lambda)'(0,\omega) \in (G_2^{-1}(\lambda - \beta),G_2^{-1}(\lambda)) \right) > 0. \]
	Therefore, $\theta_2(\lambda) \in (G_2^{-1}(\lambda - \beta),G_2^{-1}(\lambda))$ by \eqref{eq:adamvis2}--\eqref{eq:thetacom}. In particular,
	\[ \lim_{\lambda\to+\infty}\theta_2(\lambda) = +\infty. \]
	We combine this limit with Lemmas \ref{lem:sertbas} and \ref{lem:tarkan} to deduce that $\theta_2:[\beta,+\infty) \to [\theta_2(\beta),+\infty)$ is a continuous and increasing bijection. Moreover, its inverse $\theta_2^{-1}$ satisfies
	\[ 0 < \theta_2^{-1}(\theta_2(\lambda) + \epsilon/{\tilde\kappa_2^\lambda}) - \theta_2^{-1}(\theta_2(\lambda)) \le \epsilon \]
	for every $\lambda\ge\beta$ and $\epsilon \in (0,1)$, so $\theta_2^{-1}$ is locally Lipschitz continuous on its domain. This concludes the proof of the desired results regarding $\theta_2$. The analogous ones regarding $\theta_1$ follow after suitable substitutions (see \eqref{eq:lazimol}).
\end{proof}

\section{Homogenization}\label{sec:homogenization}

\subsection{Viscosity solutions}\label{sub:viscosity}

In this subsection, we recall some basic definitions and record a couple of results regarding viscosity solutions. All statements are specialized to our setting and purposes. For general background on the theory of viscosity solutions of second-order partial differential equations and its applications, we refer the reader to \cite{CIL92,FleSon06}.

We consider a HJ equation of the form
\begin{equation}\label{eq:genelHJ}
\partial_tu = a(x)\partial_{xx}^2u + G(\partial_xu) + \beta V(x),\quad (t,x)\in(0,+\infty)\times\mathbb{R},
\end{equation}
where $\beta>0$, $G:\mathbb{R}\to[0,+\infty)$ and $a,V:\mathbb{R}\to[0,1]$. It covers both \eqref{eq:birhuzur} and \eqref{eq:efhuzur}.

\begin{definition}\label{def:vis}
	A function $v\in\Con((0,+\infty)\times\mathbb{R})$ is said to be a viscosity subsolution of \eqref{eq:genelHJ} if, for every $(t_0,x_0)\in(0,+\infty)\times\mathbb{R}$ and $\varphi\in\Con^2((0,+\infty)\times\mathbb{R})$ such that $v-\varphi$ attains a local maximum at $(t_0,x_0)$, the following inequality holds:
	\[ \partial_t\varphi(t_0,x_0) \le a(x_0)\partial_{xx}^2\varphi(t_0,x_0) + G(\partial_x\varphi(t_0,x_0)) + \beta V(x_0). \]
	Similarly, a function $w\in\Con((0,+\infty)\times\mathbb{R})$ is said to be a viscosity supersolution of \eqref{eq:genelHJ} if, for every $(t_0,x_0)\in(0,+\infty)\times\mathbb{R}$ and $\varphi\in\Con^2((0,+\infty)\times\mathbb{R})$ such that $w-\varphi$ attains a local minimum at $(t_0,x_0)$, the following inequality holds:
	\[ \partial_t\varphi(t_0,x_0) \ge a(x_0)\partial_{xx}^2\varphi(t_0,x_0) + G(\partial_x\varphi(t_0,x_0)) + \beta V(x_0). \]
	Finally, a function $u\in\Con((0,+\infty)\times\mathbb{R})$ is said to be a viscosity solution of \eqref{eq:genelHJ} if it is both a viscosity subsolution and a viscosity supersolution of this equation.
\end{definition}

In the rest of this section, we repeatedly use the following comparison principle. It is covered by, e.g., \cite[Proposition 2.3]{DK17} which is a generalization of \cite[Proposition 1.4]{D19}.

\begin{proposition}\label{prop:comp}
	Assume $G\in\Con(\mathbb{R})$, $\sqrt{a}\in\Lip(\mathbb{R})$ and $V\in\UC(\mathbb{R})$. Let $v\in\UC([0,+\infty)\times\mathbb{R})$ and $w\in\UC([0,+\infty)\times\mathbb{R})$ be, respectively, a viscosity subsolution and a viscosity supersolution of \eqref{eq:genelHJ}.
	If $\{ v(t,\,\cdot\,):\,t\in[0,+\infty) \}$ is an equi-Lipschitz continuous family of functions, i.e.,
	\[ \text{there exists an $\ell>0$ such that}\ |v(t,x) - v(t,y)| \le \ell|x-y|\ \ \text{for all}\ t\in[0,+\infty)\ \text{and}\ x,y\in\mathbb{R}, \]
	or $\{ w(t,\,\cdot\,):\,t\in[0,+\infty) \}$ is an equi-Lipschitz continuous family of functions, then
	\[ \sup\{ v(t,x) - w(t,x):\,(t,x)\in[0,+\infty)\times\mathbb{R} \} = \sup\{ v(0,x) - w(0,x):\,x\in\mathbb{R} \}. \] 
\end{proposition}

The statement of Corollary \ref{cor:velinim} involves the following notion.

\begin{definition}\label{def:well}
We say that the Cauchy problem for \eqref{eq:genelHJ} is well-posed in $\UC([0,+\infty)\times\mathbb{R})$ if the following hold.
	\begin{itemize}
		\item [(i)] \underline{Existence:} For every $g\in\UC(\mathbb{R})$, \eqref{eq:genelHJ} has a viscosity solution $u\in\UC([0,+\infty)\times\mathbb{R})$ such that $u(0,\,\cdot\,) = g(\,\cdot\,)$ on $\mathbb{R}$;
		\item [(ii)] \underline{Stability:} If $u_1,u_2\in\UC([0,+\infty)\times\mathbb{R})$ are viscosity solutions of \eqref{eq:genelHJ}, then
		\[ \sup\{ |u_1(t,x) - u_2(t,x)|:\,(t,x)\in[0,+\infty)\times\mathbb{R} \} = \sup\{ |u_1(0,x) - u_2(0,x)|:\,x\in\mathbb{R} \}. \]
\end{itemize}
\end{definition}

The following result provides sufficient conditions (which are stronger versions of \eqref{eq:Gcon} and \eqref{eq:aVcon2}) for the validity of assumptions  \eqref{eq:existunique}, \eqref{eq:duzlip} and \eqref{eq:well} in Theorem \ref{thm:homlin} and Corollary \ref{cor:velinim}. It is an instance of \cite[Theorem 2.8]{DK17} whose proof is based on \cite[Theorem 3.2]{D19}.

\begin{theorem}\label{thm:suffcon}
	The Cauchy problem for \eqref{eq:birhuzur} is well-posed in $\UC([0,+\infty)\times\mathbb{R})$ for every $\omega\in\Omega$ if $G:\mathbb{R}\to[0,+\infty)$, $a(\,\cdot\,,\omega):\mathbb{R}\to(0,1]$ and $V(\,\cdot\,,\omega):\mathbb{R}\to[0,1]$ satisfy the following conditions:
	\begin{align*}
	&\text{there exist $c_1,c_2 > 0$ and $\gamma > 1$ such that}\ c_1|p|^\gamma - \frac1{c_1} \le G(p) \le c_2(|p|^\gamma + 1)\ \ \text{and}\\
	&|G(p) - G(q)| \le c_2 (|p| + |q| + 1)^{\gamma - 1}|p - q|\ \ \text{for every}\ p,q\in\mathbb{R};\\
	&\text{$\sqrt{a(\,\cdot\,,\omega)}$ and $V(\,\cdot\,,\omega)$ are in $\Lip(\mathbb{R})$ for every $\omega\in\Omega$.}
	\end{align*}
	Moreover, under these conditions, for every $\omega\in\Omega$ and $\theta\in\mathbb{R}$, the unique viscosity solution $u_\theta(\,\cdot\,,\,\cdot\,,\omega)$ of \eqref{eq:birhuzur} with the initial condition $u_\theta(0,x,\omega) = \theta x$, $x\in\mathbb{R}$, is in $\Lip([0,+\infty)\times\mathbb{R})$ with a Lipschitz constant that does not depend on $\omega$.
\end{theorem}

When the Cauchy problem for \eqref{eq:birhuzur} is well-posed in $\UC([0,+\infty)\times\mathbb{R})$ for every $\omega\in\Omega$, so is the Cauchy problem for \eqref{eq:asilhuzur} with an arbitrary $\epsilon > 0$. This is simply because $(t,x)\mapsto u(t,x,\omega)$ is a solution of \eqref{eq:birhuzur} if and only if $(t,x)\mapsto \epsilon u\left(\frac{t}{\epsilon},\frac{x}{\epsilon},\omega\right)$ is a solution of \eqref{eq:asilhuzur}.

\subsection{Locally uniform convergence for each $\theta\notin(\theta_1(\beta),\theta_2(\beta))$}\label{sub:nonflat}


\begin{lemma}\label{lem:equi}
	Assume \eqref{eq:stat} and \eqref{eq:existunique}. For every $\theta\in\mathbb{R}$, there exists an $\Omega_{\mathrm{ue}}^\theta\in\mathcal{F}$ with $\mathbb{P}(\Omega_{\mathrm{ue}}^\theta) = 1$ such that
	$\{ u_\theta^\epsilon(t,\,\cdot\,,\omega):\,\epsilon\in(0,1],\ t\in[0,+\infty),\ \omega\in\Omega_{\mathrm{ue}}^\theta \}$ is a uniformly equicontinuous family of functions.
\end{lemma}

\begin{proof}	
	Fix any $\theta\in\mathbb{R}$. For every $\omega\in\Omega$, the function $m_\theta(\,\cdot\,,\omega):[0,+\infty)\to[0,+\infty)$, defined by
	\[ m_\theta(\delta,\omega) = \sup_{t\in[0,+\infty)}\sup_{|x-y|\le\delta}|u_\theta(t,x,\omega) - u_\theta(t,y,\omega)|, \]
	is uniformly continuous by \eqref{eq:existunique}. For every $z\in\mathbb{R}$, $m_\theta(\delta,\omega) = m_\theta(\delta,\tau_z\omega)$ by \eqref{eq:stat}, \eqref{eq:existunique} and the observation that
	\[ u_\theta(t,x+z,\omega) - u_\theta(t,y+z,\omega) = u_\theta(t,x,\tau_z\omega) - u_\theta(t,y,\tau_z\omega) \]
	for all $t\in[0,+\infty)$ and $x,y\in\mathbb{R}$. Therefore, by ergodicity (and the countability of $\mathbb{Q}$), there exists an $\Omega_{\mathrm{ue}}^\theta\in\mathcal{F}$ with $\mathbb{P}(\Omega_{\mathrm{ue}}^\theta) = 1$ and a function $\overline m_\theta:[0,+\infty)\cap\mathbb{Q}\to[0,+\infty)$ such that $m_\theta(\delta,\omega) = \overline m_\theta(\delta)$ for all $\delta\in[0,+\infty)\cap\mathbb{Q}$ and $\omega\in\Omega_{\mathrm{ue}}^\theta$. It follows that $\overline m_\theta$ is uniformly continuous on its domain and the uniformly continuous extension of $\overline m_\theta$ to $[0,+\infty)$ (still denoted by $\overline m_\theta$) satisfies
	\[ |u_\theta(t,x,\omega) - u_\theta(t,y,\omega)| \le m_\theta(|x-y|,\omega) = \overline m_\theta(|x-y|) \]
	for all $t\in[0,+\infty)$, $x,y\in\mathbb{R}$ and $\omega\in\Omega_{\mathrm{ue}}^\theta$. Finally, for every $\epsilon\in(0,1]$, by letting $k = \lceil \frac1{\epsilon} \rceil$ and noting that $\frac1{\epsilon} \le k < \frac1{\epsilon} + 1 \le \frac{2}{\epsilon}$, we obtain the following inequality:
	\begin{align*}
		|u_\theta^\epsilon(t,x,\omega) - u_\theta^\epsilon(t,y,\omega)| &= \epsilon \left| u_\theta\left(\frac{t}{\epsilon},\frac{x}{\epsilon},\omega\right) - u_\theta\left(\frac{t}{\epsilon},\frac{y}{\epsilon},\omega\right) \right|\\
		&\le \epsilon k \overline m_\theta\left( \frac{|x-y|}{\epsilon k} \right) \le 2\overline m_\theta(|x-y|).\qedhere
	\end{align*}
\end{proof}

\begin{lemma}\label{lem:isbu}
	Assume \eqref{eq:Gcon}--\eqref{eq:fruit} and \eqref{eq:existunique}--\eqref{eq:aVcon2}.
	\begin{itemize}
		\item [(a)]	For every $\theta\in(-\infty,\theta_1(\beta)]$, there exists an $\Omega_0^\theta\in\mathcal{F}$ with $\mathbb{P}(\Omega_0^\theta) = 1$ such that, for every $\omega\in\Omega_0^\theta$ and $T,L>0$,
		\[ \lim_{\epsilon\to0}\sup_{t\in[0,T]}\sup_{x\in[-L,L]} |u_\theta^\epsilon(t,x,\omega) - t\theta_1^{-1}(\theta) - \theta x| = 0. \]
		\item [(b)] For every $\theta\in[\theta_2(\beta),+\infty)$, there exists an $\Omega_0^\theta\in\mathcal{F}$ with $\mathbb{P}(\Omega_0^\theta) = 1$ such that, for every $\omega\in\Omega_0^\theta$ and $T,L>0$,
		\[ \lim_{\epsilon\to0}\sup_{t\in[0,T]}\sup_{x\in[-L,L]} |u_\theta^\epsilon(t,x,\omega) - t\theta_2^{-1}(\theta) - \theta x| = 0. \]
	\end{itemize}
\end{lemma}

\begin{proof}
	We prove part (b). (The proof of part (a) is similar.) Fix any $\theta\in[\theta_2(\beta),+\infty)$ and let $\lambda = \theta_2^{-1}(\theta)$. It follows immediately from Theorem \ref{thm:staticHJ} that, for every $\omega\in\Omega$,
	\[ u_2^\lambda(t,x,\omega) = t\lambda + F_2^\lambda(x,\omega) \]
	gives a solution of \eqref{eq:birhuzur} in $\Lip\cap\Con^2([0,+\infty)\times\mathbb{R})$.
	
	For every $\delta\in(0,1)$, define $v_{2,\delta}^\lambda(\,\cdot\,,\,\cdot\,,\omega)$ by
	\begin{align*}
		v_{2,\delta}^\lambda(t,x,\omega) &= t(\lambda - (\kappa + 1)\delta) + F_2^\lambda(x,\omega) - \delta\psi(x) - K\\
					  &= u_2^\lambda(t,x,\omega) - t(\kappa + 1)\delta - \delta\psi(x) - K,
	\end{align*}
	where $\kappa$ is a Lipschitz constant for $G$ on the interval $[G_2^{-1}(\lambda - \beta) - 1,G_2^{-1}(\lambda) + 1]$,
	\[ \psi(x) = \frac{2}{\pi}\int_0^x\arctan(y)dy \]
	which satisfies
	\begin{equation}\label{eq:sofra1}
		0 \le \psi''(\cdot) \le 1,\ \ -1 \le \psi'(\cdot) \le 1,
	\end{equation}
	\begin{equation}\label{eq:sofra2}
		\lim_{x\to-\infty}\psi'(x) = -1\ \ \text{and}\ \ \lim_{x\to+\infty}\psi'(x) = 1,
	\end{equation}
	and $K>0$ is to be determined. Note that, for every $(t,x)\in(0,+\infty)\times\mathbb{R}$,
	\begin{align}
		&\quad\ a(x,\omega)\partial_{xx}^2v_{2,\delta}^\lambda(t,x,\omega) + G(\partial_xv_{2,\delta}^\lambda(t,x,\omega)) + \beta V(x,\omega)\nonumber\\
		&= a(x,\omega)(\partial_{xx}^2u_2^\lambda(t,x,\omega) - \delta\psi''(x)) + G(\partial_xu_2^\lambda(t,x,\omega) - \delta\psi'(x)) + \beta V(x,\omega)\nonumber\\
		&\ge a(x,\omega)\partial_{xx}^2u_2^\lambda(t,x,\omega) - \delta + G(\partial_xu_2^\lambda(t,x,\omega)) - \kappa\delta + \beta V(x,\omega)\label{eq:kullan}\\
		&= \partial_tu_2^\lambda(t,x,\omega) - (\kappa + 1)\delta = \partial_tv_{2,\delta}^\lambda(t,x,\omega).\nonumber
	\end{align}
	The inequality in \eqref{eq:kullan} follows from \eqref{eq:sofra1} and the following bounds due to \eqref{eq:adamvis2}: 
	\begin{align*}
		G_2^{-1}(\lambda - \beta) - 1 \le G_2^{-1}(\lambda - \beta) - \delta &\le (F_2^\lambda)'(x,\omega) - \delta\\
												    &< (F_2^\lambda)'(x,\omega) + \delta < G_2^{-1}(\lambda) + \delta \le G_2^{-1}(\lambda) + 1.
	\end{align*}
	Hence, $v_{2,\delta}^\lambda(\,\cdot\,,\,\cdot\,,\omega)$ is a subsolution of \eqref{eq:birhuzur} in $\Lip\cap\Con^2([0,+\infty)\times\mathbb{R})$.
	
	For $\mathbb{P}$-a.e.\ $\omega$,
	\[ \lim_{x\to-\infty}\frac1{x}v_{2,\delta}^\lambda(0,x,\omega) = \theta_2(\lambda) + \delta = \theta + \delta \quad\text{and}\quad \lim_{x\to+\infty}\frac1{x}v_{2,\delta}^\lambda(0,x,\omega) = \theta_2(\lambda) - \delta = \theta - \delta \]
	by \eqref{eq:ardil} and \eqref{eq:sofra2}. Therefore,
	\[ v_{2,\delta}^\lambda(0,x,\omega) \le \theta x = u_\theta(0,x,\omega) \]
	for every $x\in\mathbb{R}$ when $K = K(\theta,\delta,\omega) > 0$ is sufficiently large. By the comparison principle in Proposition \ref{prop:comp},
	\[ v_{2,\delta}^\lambda(t,x,\omega) \le u_\theta(t,x,\omega)\ \text{for every $(t,x)\in[0,+\infty)\times\mathbb{R}$}. \]
	In particular,
	\[ \liminf_{\epsilon\to0}u_\theta^\epsilon(1,0,\omega) = \liminf_{\epsilon\to0}\epsilon u_\theta\left(\frac1{\epsilon},0,\omega\right) \ge \lim_{\epsilon\to0}\epsilon v_{2,\delta}^\lambda\left(\frac1{\epsilon},0,\omega\right) = \lambda - (\kappa + 1)\delta. \]

	Similarly,
	\[ w_{2,\delta}^\lambda(t,x,\omega) = t(\lambda + (\kappa + 1)\delta) + F_2^\lambda(x,\omega) + \delta\psi(x) + K \]
	defines a supersolution of \eqref{eq:birhuzur} in $\Lip\cap\Con^2([0,+\infty)\times\mathbb{R})$, and, for $\mathbb{P}$-a.e.\ $\omega$,
	\[ w_{2,\delta}^\lambda(0,x,\omega) \ge \theta x = u_\theta(0,x,\omega) \]
	for every $x\in\mathbb{R}$ when $K = K(\theta,\delta,\omega) > 0$ is sufficiently large. By the comparison principle in Proposition \ref{prop:comp},
	\[ \limsup_{\epsilon\to0}u_\theta^\epsilon(1,0,\omega) = \limsup_{\epsilon\to0}\epsilon u_\theta\left(\frac1{\epsilon},0,\omega\right) \le \lim_{\epsilon\to0}\epsilon w_{2,\delta}^\lambda\left(\frac1{\epsilon},0,\omega\right) = \lambda + (\kappa + 1)\delta. \]
	Since $\delta\in(0,1)$ is arbitrary, we deduce that
	\begin{equation}\label{eq:nokta}
		\lim_{\epsilon\to0}u_\theta^\epsilon(1,0,\omega) = \lambda = \theta_2^{-1}(\theta)\ \ \text{for $\mathbb{P}$-a.e.\ $\omega$.}
	\end{equation}
	
	Finally, the desired locally uniform convergence follows from \eqref{eq:nokta} and Lemma \ref{lem:equi} by a general argument involving Egorov's theorem and the Birkhoff ergodic theorem. See \cite[pp.\ 1501--1502]{KRV06} or \cite[Lemma 4.1]{DK17} which is based on \cite[Lemma 2.4]{AT14}.
\end{proof}

\subsection{Locally uniform convergence for each $\theta\in(\theta_1(\beta),\theta_2(\beta))$}\label{sub:flat}

In this subsection, we will take $\lambda = \beta$ and denote the derivatives of the unique solutions in Theorem \ref{thm:staticHJ} by $f_i^\beta = (F_i^\beta)'$, $i\in\{1,2\}$. We will also use the following notation:
\begin{equation}\label{eq:es}
	s(x,\omega) = \int_0^x\frac{dy}{a(y,\omega)}.
\end{equation}

\begin{lemma}\label{lem:kayahan}
	Given any $\omega\in\Omega$, $\delta\in(0,\beta)$ and $L_1,L_2\in\mathbb{R}$ such that 
	\[ s(L_2,\omega) - s(L_1,\omega) > \frac1{\delta}\left(G_2^{-1}(\beta) - G_1^{-1}(\beta)\right) \quad\text{and}\quad \beta V(\,\cdot\,,\omega) \ge \beta - \delta\ \text{on}\ [L_1,L_2], \]
	we have the following implications for every $x_1,x_2\in[L_1,L_2]$.
	\begin{itemize}
	    \item[(i)] If
		\begin{equation}\label{eq:ekal1}
			s(x_1,\omega) - s(L_1,\omega) > -\frac1{\delta}G_1^{-1}(\beta),
		\end{equation}
		then there is a $z_1\in(L_1,x_1]$ such that
		\[ G_1(f_1^\beta(z_1,\omega)) \le 2\delta \quad\text{and}\quad s(x_1,\omega) - s(z_1,\omega) \le -\frac1{\delta}G_1^{-1}(\beta). \]

		\item [(ii)] If
		\begin{equation}\label{eq:ekal2}
			s(L_2,\omega) - s(x_2,\omega) > \frac1{\delta}G_2^{-1}(\beta),
		\end{equation}
		then there is a $z_2\in[x_2,L_2)$ such that
		\[ G_2(f_2^\beta(z_2,\omega)) \le 2\delta \quad\text{and}\quad s(z_2,\omega) - s(x_2,\omega) \le \frac1{\delta}G_2^{-1}(\beta). \]
	\end{itemize}
\end{lemma}	

\begin{proof}
	We prove the second implication. If $G_2(f_2^\beta(x_2,\omega)) \le 2\delta$, then we can simply take $z_2 = x_2$. Otherwise, for any $z \in [x_2,L_2]$ such that $G_2(f_2^\beta(x,\omega)) \ge 2\delta$ holds for all $x\in[x_2,z]$,
	the equality
	\[ a(x,\omega)(f_2^\beta)'(x,\omega) + G_2(f_2^\beta(x,\omega)) + \beta V(x,\omega) = \beta \]
	yields
	\[ a(x,\omega)(f_2^\beta)'(x,\omega) \le -\delta \]
		for all $x\in[x_2,z]$, and
		\[ -G_2^{-1}(\beta) \le f_2^\beta(z,\omega) - f_2^\beta(x_2,\omega) = \int_{x_2}^{z} (f_2^\beta)'(y,\omega)dy \le -\delta \int_{x_2}^{z}\frac{dy}{a(y,\omega)} = -\delta \left[s(z,\omega) - s(x_2,\omega)\right]. \]
		The first inequality uses the bounds in \eqref{eq:adamvis2}. Therefore,
		\[ z_2 := \sup\{ z\in [x_2,L_2]:\,G_2(f_2^\beta(x,\omega)) \ge 2\delta\ \text{for all}\ x\in[x_2,z] \} \]
		satisfies
		\[ s(z_2,\omega) - s(x_2,\omega) \le \frac1{\delta}G_2^{-1}(\beta). \]
		We recall \eqref{eq:ekal2} and deduce that $z_2\in[x_2,L_2)$ and $G_2(f_2^\beta(z_2,\omega)) = 2\delta$. This concludes the proof of the second implication. The first implication is proved similarly.
\end{proof}

\begin{lemma}\label{lem:LBUB}
	Assume \eqref{eq:Gcon}--\eqref{eq:fruit} and \eqref{eq:existunique}--\eqref{eq:scaledhill}. There exists an $\overline\Omega_0\in\mathcal{F}$ with $\mathbb{P}(\overline\Omega_0) = 1$ such that
	\[ \lim_{\epsilon\to0} u_\theta^\epsilon(1,0,\omega) = \beta \]
	for all $\theta\in(\theta_1(\beta),\theta_2(\beta))$ and $\omega\in\overline\Omega_0$. Moreover, given any $\theta\in(\theta_1(\beta),\theta_2(\beta))$, there exists an $\Omega_0^\theta\in\mathcal{F}$ with $\mathbb{P}(\Omega_0^\theta) = 1$ such that, for every $\omega\in\Omega_0^\theta$ and $T,L>0$,
	\begin{equation}\label{eq:locunif3}
		\lim_{\epsilon\to0}\sup_{t\in[0,T]}\sup_{x\in[-L,L]} |u_\theta^\epsilon(t,x,\omega) - t\beta - \theta x| = 0.
	\end{equation}
\end{lemma}

\begin{proof}
By the scaled hill condition \eqref{eq:scaledhill}, there exists an $\Omega_{\mathrm{sh}}\in\mathcal{F}$ with $\mathbb{P}(\Omega_{\mathrm{sh}}) = 1$ such that, for every $\omega\in\Omega_{\mathrm{sh}}$, $\delta\in(0,\beta)$ and $C > 0$, there is an interval $[L_1,L_2]$ such that
\begin{equation}\label{eq:shillbu}
	s(L_2,\omega) - s(L_1,\omega) \ge C \qquad\text{and}\qquad \beta V(\,\cdot\,,\omega) \ge \beta - \delta\ \text{on}\ [L_1,L_2]
\end{equation}
with the notation in \eqref{eq:es}. This follows from the observation that it suffices to consider $\delta\in(0,\beta)\cap\mathbb{Q}$ and $C \in \mathbb{N} = \{1,2,3,\ldots\}$.

\subsubsection*{Asymptotic lower bound at $(1,0)$}

For every $\omega\in\Omega_{\mathrm{sh}}$, $\delta\in(0,\beta)$ and $C > \frac1{\delta}\left(G_2^{-1}(\beta) - G_1^{-1}(\beta)\right)$, take an interval $[L_1,L_2]$ that satisfies \eqref{eq:shillbu}. Let $x_1 = L_2$ and $x_2=L_1$. By Lemma \ref{lem:kayahan}, there exist $z_1\in(L_1,L_2]$ and $z_2\in[L_1,L_2)$ such that
\begin{equation}\label{eq:ekmek}
s(z_1,\omega) - s(z_2,\omega) \ge C - \frac1{\delta}\left(G_2^{-1}(\beta) - G_1^{-1}(\beta)\right)
\end{equation}
and
\begin{equation}\label{eq:gercekler}
G(f_i^\beta(z_i,\omega)) \le 2\delta,\quad i\in\{1,2\}.
\end{equation}
In particular, $L_1 <z_2 < z_1 < L_2$. It follows from 
\[ a(x,\omega)(f_i^\beta)'(x,\omega) + G(f_i^\beta(x,\omega)) + \beta V(x,\omega) = \beta,\quad i\in\{1,2\}, \]
and \eqref{eq:gercekler} that
\begin{equation}\label{eq:gercekler2}
-2\delta \le a(z_i,\omega)(f_i^\beta)'(z_i,\omega) \le \delta, \quad i\in\{1,2\}.
\end{equation}

When $C$ is sufficiently large, there exists a $g(\,\cdot\,,\omega)\in\Con^1(\mathbb{R})$ that satisfies the following conditions:
\begin{equation}\label{eq:phi1}
	\begin{aligned}
		&g(z_1,\omega) = f_1^\beta(z_1,\omega),\quad g'(z_1,\omega) = (f_1^\beta)'(z_1,\omega),\\
		&g(z_2,\omega) = f_2^\beta(z_2,\omega),\quad g'(z_2,\omega) = (f_2^\beta)'(z_2,\omega);
	\end{aligned}
\end{equation}
\begin{equation}\label{eq:phi2}
	G(g(x,\omega)) \le 3\delta \quad\text{and}\quad -2\delta \le a(x,\omega)g'(x,\omega) \le \delta \quad \text{for all}\ x\in[z_2,z_1].
\end{equation}
Indeed, for every $\eta \in (0,1)$, we can take a function $\tilde g(\,\cdot\,,\omega)$ of the form
\[ \tilde g(x,\omega) = c_1 + f_2^\beta(z_2,\omega) - (c_2 + f_2^\beta(z_2,\omega) - f_1^\beta(z_1,\omega))\frac{s(x,\omega) - s(z_2,\omega)}{s(z_1,\omega) - s(z_2,\omega)} \]
with appropriately chosen $c_1,c_2 \in(-\eta,\eta)$ and modify it around $z_1$ and $z_2$ to match the conditions in \eqref{eq:phi1}. In case the modification requires an overshoot in the function value, the error margin in the first inequality in \eqref{eq:phi2} is larger than the one in \eqref{eq:gercekler}. The derivative bounds in \eqref{eq:phi2} are consistent with those in \eqref{eq:gercekler2}. To stay within these bounds on $(z_2,z_1)$, it suffices to make the right-hand side of \eqref{eq:ekmek} greater than $\frac1{\delta}\left(G_2^{-1}(3\delta) - G_1^{-1}(3\delta) + 1\right)$.

Construct $F_{2,1}^\beta(\,\cdot\,,\omega)$ by setting $F_{2,1}^\beta(0,\omega) = 0$ and
\begin{equation}\label{eq:sekil}
	(F_{2,1}^\beta)'(x,\omega) = \begin{cases} f_2^\beta(x,\omega)&\text{if}\ x \le z_2,\\
												      g (x,\omega)&\text{if}\ z_2 < x < z_1,\\
											   f_1^\beta(x,\omega)&\text{if}\ x \ge z_1.\end{cases}
\end{equation}
Note that $F_{2,1}^\beta(\,\cdot\,,\omega)\in\Lip\cap\Con^2(\mathbb{R})$ by \eqref{eq:phi1}--\eqref{eq:phi2}. Moreover, since
\[ \beta - 3\delta = -2\delta + 0 + (\beta - \delta) \le a(x,\omega)g'(x,\omega) + G(g(x,\omega)) + \beta V(x,\omega) \le \delta + 3\delta + \beta = \beta + 4\delta \]
for every $x\in[z_2,z_1]\subset[L_1,L_2]$ by \eqref{eq:phi2},
\[ \beta - 3\delta \le a(x,\omega)(F_{2,1}^\beta)''(x,\omega) + G((F_{2,1}^\beta)'(x,\omega)) + \beta V(x,\omega) \le \beta + 4\delta \]
for every $x\in\mathbb{R}$. It follows immediately that $v_{0,\delta}^\beta(\,\cdot\,,\,\cdot\,,\omega)$, defined by
\[ v_{0,\delta}^\beta(t,x,\omega) = t(\beta - 3\delta) + F_{2,1}^\beta(x,\omega) - K, \]
where $K > 0$ is to be determined, is a subsolution of \eqref{eq:birhuzur} in $\Lip\cap\Con^2([0,+\infty)\times\mathbb{R})$.

By the definitions in \eqref{eq:thetacom} and the Birkhoff ergodic theorem, there exists an $\Omega_{1,2}\in\mathcal{F}$ with $\mathbb{P}(\Omega_{1,2}) = 1$ such that, for every $\omega\in\Omega_{1,2}$,
\begin{equation}\label{eq:alrayyan}
	\lim_{x\to\pm\infty} \frac1{x}F_1^\beta(x,\omega) = \theta_1(\beta) \quad\text{and}\quad \lim_{x\to\pm\infty} \frac1{x}F_2^\beta(x,\omega) = \theta_2(\beta).
\end{equation}
Let $\overline\Omega_0 = \Omega_{\mathrm{sh}}\cap\Omega_{1,2}$ and note that $\mathbb{P}(\overline\Omega_0) = 1$. For every $\omega\in\overline\Omega_0$,
\[ \lim_{x\to-\infty}\frac1{x}v_{0,\delta}^\beta(0,x,\omega) = \theta_2(\beta) \quad\text{and}\quad \lim_{x\to+\infty}\frac1{x}v_{0,\delta}^\beta(0,x,\omega) = \theta_1(\beta) \]
by \eqref{eq:sekil}--\eqref{eq:alrayyan}. Therefore, given any $\theta\in(\theta_1(\beta),\theta_2(\beta))$ and $\omega\in\overline\Omega_0$,
\[ v_{0,\delta}^\beta(0,x,\omega) \le \theta x = u_\theta(0,x,\omega) \]
for every $x\in\mathbb{R}$ when $K = K(\theta,\delta,\omega) > 0$ is sufficiently large. 
By the comparison principle in Proposition \ref{prop:comp},
\[ v_{0,\delta}^\beta(t,x,\omega) \le u_\theta(t,x,\omega)\ \text{for every $(t,x)\in[0,+\infty)\times\mathbb{R}$}. \]
In particular,
\[ \liminf_{\epsilon\to0}u_\theta^\epsilon(1,0,\omega) = \liminf_{\epsilon\to0}\epsilon u_\theta\left(\frac1{\epsilon},0,\omega\right) \ge \lim_{\epsilon\to0}\epsilon v_{0,\delta}^\beta\left(\frac1{\epsilon},0,\omega\right) = \beta - 3\delta. \]
Since $\delta\in(0,\beta)$ is arbitrary, we deduce that
\begin{equation}\label{eq:flatLB}
	\liminf_{\epsilon\to0}u_\theta^\epsilon(1,0,\omega) \ge \beta\ \ \text{for all $\theta\in(\theta_1(\beta),\theta_2(\beta))$ and $\omega\in\overline\Omega_0$.}
\end{equation}

\subsubsection*{Asymptotic upper bound at $(1,0)$}

For every $\omega\in\Omega_{\mathrm{sh}}$, $\delta\in(0,\beta)$ and $C > \frac{2}{\delta}\left(G_2^{-1}(\beta) - G_1^{-1}(\beta)\right)$, take an interval $[L_1,L_2]$ that satisfies \eqref{eq:shillbu}. Fix $x_1,x_2\in[L_1,L_2]$ such that \eqref{eq:ekal1}, \eqref{eq:ekal2} and
\[ s(x_2,\omega) - s(x_1,\omega) \ge \frac{C}{2} \]
are satisfied. By Lemma \ref{lem:kayahan}, there exist $z_1,z_2\in\mathbb{R}$ such that $L_1 < z_1 \le x_1 < x_2 \le z_2 < L_2$,
\[ s(z_2,\omega) - s(z_1,\omega) \ge s(x_2,\omega) - s(x_1,\omega) \ge \frac{C}{2} \]
and
\[ G(f_i^\beta(z_i,\omega)) \le 2\delta,\quad i\in\{1,2\}. \]
As we argued in the proof of the lower bound, for $C$ sufficiently large, there exists a $g(\,\cdot\,,\omega)\in\Con^1(\mathbb{R})$ that satisfies \eqref{eq:phi1}
as well as
\begin{equation}\label{eq:phi3}
G(g(x,\omega)) \le 3\delta \quad\text{and}\quad -2\delta \le a(x,\omega)g'(x,\omega) \le \delta \quad \text{for all}\ x\in[z_1,z_2].
\end{equation}

Construct $F_{1,2}^\beta(\,\cdot\,,\omega)$ by setting $F_{1,2}^\beta(0,\omega) = 0$ and
\begin{equation}\label{eq:semal}
	(F_{1,2}^\beta)'(x,\omega) = \begin{cases} f_1^\beta(x,\omega)&\text{if}\ x \le z_1,\\
													   g(x,\omega)&\text{if}\ z_1 < x < z_2,\\
											   f_2^\beta(x,\omega)&\text{if}\ x \ge z_2.\end{cases}
\end{equation}
Note that $F_{1,2}^\beta(\,\cdot\,,\omega)\in\Lip\cap\Con^2(\mathbb{R})$ by \eqref{eq:phi1} and \eqref{eq:phi3}. Moreover, since
\[ \beta - 3\delta = -2\delta + 0 + (\beta - \delta) \le a(x,\omega)g'(x,\omega) + G(g(x,\omega)) + \beta V(x,\omega) \le \delta + 3\delta + \beta = \beta + 4\delta \]
for every $x\in[z_1,z_2]\subset[L_1,L_2]$ by \eqref{eq:phi3},
\[ \beta - 3\delta \le a(x,\omega)(F_{1,2}^\beta)''(x,\omega) + G((F_{1,2}^\beta)'(x,\omega)) + \beta V(x,\omega) \le \beta + 4\delta \]
for every $x\in\mathbb{R}$. It follows immediately that $w_{0,\delta}^\beta(\,\cdot\,,\,\cdot\,,\omega)$, defined by
\[ w_{0,\delta}^\beta(t,x,\omega) = t(\beta + 4\delta) + F_{1,2}^\beta(x,\omega) + K, \]
where $K > 0$ is to be determined, is a supersolution of \eqref{eq:birhuzur} in $\Lip\cap\Con^2([0,+\infty)\times\mathbb{R})$.
					  
For every $\omega\in\overline\Omega_0 = \Omega_{\mathrm{sh}}\cap\Omega_{1,2}$ (with the notation in the proof of the lower bound),
\[ \lim_{x\to-\infty}\frac1{x}w_{0,\delta}^\beta(0,x,\omega) = \theta_1(\beta) \quad\text{and}\quad \lim_{x\to+\infty}\frac1{x}w_{0,\delta}^\beta(0,x,\omega) = \theta_2(\beta) \]
by \eqref{eq:alrayyan} and \eqref{eq:semal}. Therefore, given any $\theta\in(\theta_1(\beta),\theta_2(\beta))$ and $\omega\in\overline\Omega_0$,
\[ w_{0,\delta}^\beta(0,x,\omega) \ge \theta x = u_\theta(0,x,\omega) \]
for every $x\in\mathbb{R}$ when $K = K(\theta,\delta,\omega) > 0$ is sufficiently large. 
By the comparison principle in Proposition \ref{prop:comp},
\[ w_{0,\delta}^\beta(t,x,\omega) \ge u_\theta(t,x,\omega)\ \text{for every $(t,x)\in[0,+\infty)\times\mathbb{R}$}. \]
In particular,
\[ \limsup_{\epsilon\to0}u_\theta^\epsilon(1,0,\omega) = \limsup_{\epsilon\to0}\epsilon u_\theta\left(\frac1{\epsilon},0,\omega\right) \le \lim_{\epsilon\to0}\epsilon w_{0,\delta}^\beta\left(\frac1{\epsilon},0,\omega\right) = \beta + 4\delta. \]
Since $\delta\in(0,\beta)$ is arbitrary, we deduce that
\begin{equation}\label{eq:flatUB}
\limsup_{\epsilon\to0}u_\theta^\epsilon(1,0,\omega) \le \beta\ \ \text{for all $\theta\in(\theta_1(\beta),\theta_2(\beta))$ and $\omega\in\overline\Omega_0$.}
\end{equation}

\subsubsection*{Pointwise convergence at $(1,0)$}

Combining \eqref{eq:flatLB} and \eqref{eq:flatUB}, we conclude that
\[ \lim_{\epsilon\to0}u_\theta^\epsilon(1,0,\omega) = \beta \]
for all $\theta\in(\theta_1(\beta),\theta_2(\beta))$ and $\omega\in\overline\Omega_0$.

\subsubsection*{Locally uniform convergence} 

Fix any $\theta\in(\theta_1(\beta),\theta_2(\beta))$. Recall from Lemma \ref{lem:equi} that there exists an $\Omega_{\mathrm{ue}}^\theta\in\mathcal{F}$ with $\mathbb{P}(\Omega_{\mathrm{ue}}^\theta) = 1$ such that $\{ u_\theta^\epsilon(t,\,\cdot\,,\omega):\,\epsilon\in(0,1],\ t\in[0,+\infty),\ \omega\in\Omega_{\mathrm{ue}}^\theta \}$ is a uniformly equicontinuous family of functions.
By the general argument (involving Egorov's theorem and the Birkhoff ergodic theorem) we cited at the end of the proof of Lemma \ref{lem:isbu}, there exists an $\Omega_0^\theta\subset\overline\Omega_0\cap\Omega_{\mathrm{ue}}^\theta$ with $\mathbb{P}\left(\left(\overline\Omega_0\cap\Omega_{\mathrm{ue}}^\theta\right)\setminus\Omega_0^\theta\right) = 0$ (which implies $\mathbb{P}(\Omega_0^\theta) = 1$) such that \eqref{eq:locunif3} holds for every $\omega\in\Omega_0^\theta$ and $T,L>0$.
\end{proof}

\subsection{Completing the proofs of the homogenization results}\label{sub:homproofs}

\begin{proof}[Proof of Theorem \ref{thm:homlin}]
	By Theorem \ref{thm:thetasOK}, $\Ham\in\Liploc(\mathbb{R})$ and it is coercive. Therefore, the Cauchy problem for \eqref{eq:efhuzur} is well-posed in $\UC([0,+\infty)\times\mathbb{R})$ (see, e.g., \cite[Theorem 2.5]{DK17}). For every $\theta\in\mathbb{R}$, observe that the unique (classical and hence viscosity) solution $\overline u_\theta$ of \eqref{eq:efhuzur} with the initial condition $\overline u_\theta(x) = \theta x$, $x\in\mathbb{R}$, is given by
	\[ \overline u_\theta(t,x) = t\Ham(\theta) + \theta x. \]
	
	Let
	\[ \Omega_0 = \bigcap_{\theta\in\mathbb{Q}}\Omega_0^\theta \]
	with $\Omega_0^\theta\in\mathcal{F}$ provided in Lemma \ref{lem:isbu} and Lemma \ref{lem:LBUB} when $\theta\notin(\theta_1(\beta),\theta_2(\beta))$ and $\theta\in(\theta_1(\beta),\theta_2(\beta))$, respectively. Note that $\mathbb{P}(\Omega_0) = 1$ and, for every $\omega\in\Omega_0$ and $\theta\in\mathbb{Q}$, as $\epsilon\to0$, $u_\theta^\epsilon(\,\cdot\,,\,\cdot\,,\omega)$ converges locally uniformly on $[0,+\infty)\times\mathbb{R}$ to $\overline u_\theta$. It remains to generalize this statement to all $\theta\in\mathbb{R}$.
	
	Fix any $\theta\in\mathbb{R}$. For every $\omega\in\Omega$ and $\delta\in(0,1)$, define $v_{\theta,\delta}(\,\cdot\,,\,\cdot\,,\omega)$ and $w_{\theta,\delta}(\,\cdot\,,\,\cdot\,,\omega)$ by
	\begin{equation}\label{eq:kamp}
	\begin{aligned}
		v_{\theta,\delta}(t,x,\omega) &= u_\theta(t,x,\omega) - t(\kappa_\theta(\omega) + 1)\delta - \delta\psi(x) - K\quad\text{and}\\
		w_{\theta,\delta}(t,x,\omega) &= u_\theta(t,x,\omega) + t(\kappa_\theta(\omega) + 1)\delta + \delta\psi(x) + K,
	\end{aligned}
	\end{equation}
	where $\kappa_\theta(\omega)$ is a Lipschitz constant for $G$ on the interval $[-\ell_\theta(\omega) - 1,\ell_\theta(\omega) + 1]$ which in turn involves the Lipschitz constant $\ell_\theta(\omega)$ in \eqref{eq:duzlip},
	\[ \psi(x) = \frac{2}{\pi}\int_0^x\arctan(y)dy \]
	which satisfies \eqref{eq:sofra1}--\eqref{eq:sofra2} from the proof of Lemma \ref{lem:isbu}, and $K>0$ is to be determined. Let us check that $v_{\theta,\delta}(\,\cdot\,,\,\cdot\,,\omega)$ is a viscosity subsolution of \eqref{eq:birhuzur}. For every $(t_0,x_0)\in(0,+\infty)\times\mathbb{R}$ and $\varphi\in\Con^2((0,+\infty)\times\mathbb{R})$ such that $v_{\theta,\delta}(\,\cdot\,,\,\cdot\,,\omega)-\varphi$ attains a local maximum at $(t_0,x_0)$, define $\tilde\varphi\in\Con^2((0,+\infty)\times\mathbb{R})$ by
	\[ \tilde\varphi(t,x) = \varphi(t,x) +  t(\kappa_\theta(\omega) + 1)\delta + \delta\psi(x) + K \]
	and note that $u_\theta(\,\cdot\,,\,\cdot\,,\omega) - \tilde\varphi = v_{\theta,\delta}(\,\cdot\,,\,\cdot\,,\omega)-\varphi$.
	Therefore,
	\begin{align*}
	&\quad\ a(x_0,\omega)\partial_{xx}^2\varphi(t_0,x_0) + G(\partial_x\varphi(t_0,x_0)) + \beta V(x_0,\omega)\\
	&= a(x_0,\omega)(\partial_{xx}^2\tilde\varphi(t_0,x_0) - \delta\psi''(x)) + G(\partial_x\tilde\varphi(t_0,x_0) - \delta\psi'(x)) + \beta V(x_0,\omega)\\
	&\ge a(x_0,\omega)\partial_{xx}^2\tilde\varphi(t_0,x_0) - \delta + G(\partial_x\tilde\varphi(t_0,x_0)) - \kappa_\theta(\omega)\delta + \beta V(x_0,\omega)\\
	&\ge \partial_t\tilde\varphi(t_0,x_0) - (\kappa_\theta(\omega) + 1)\delta = \partial_t\varphi(t_0,x_0).
	\end{align*}
	Similarly, $w_{\theta,\delta}(\,\cdot\,,\,\cdot\,,\omega)$ is a viscosity supersolution of \eqref{eq:birhuzur}.
	
	Choose any $\theta'\in\mathbb{Q}$ such that $|\theta - \theta'| < \frac{\delta}{2}$. It follows from \eqref{eq:sofra2} that, when $K = K(\delta) > 0$ is sufficiently large,
	\[ v_{\theta,\delta}(0,x,\omega) = \theta x - \delta\psi(x) - K \le u_{\theta'}(0,x,\omega) = \theta' x \le \theta x + \delta\psi(x) + K = w_{\theta,\delta}(0,x,\omega) \]
	for every $x\in\mathbb{R}$. By the comparison principle in Proposition \ref{prop:comp},
	\[ v_{\theta,\delta}(t,x,\omega) \le u_{\theta'}(t,x,\omega) \le w_{\theta,\delta}(t,x,\omega)\ \text{for every $(t,x)\in[0,+\infty)\times\mathbb{R}$}. \]	
	We combine these inequalities with the definitions in \eqref{eq:kamp} and deduce that
	\[ |u_\theta(t,x,\omega) - u_{\theta'}(t,x,\omega)| \le t(\kappa_\theta(\omega) + 1)\delta + \delta|x| + K \] 
	for every $\omega\in\Omega$ and $(t,x)\in[0,+\infty)\times\mathbb{R}$.
	
	Finally, for every $\omega\in\Omega_0$ and $T,L>0$,
	\begin{align*}
		&\quad\, \limsup_{\epsilon\to0}\sup_{t\in[0,T]}\sup_{x\in[-L,L]} |u_\theta^\epsilon(t,x,\omega) - t\Ham(\theta) - \theta x|\\
		&\le \limsup_{\epsilon\to0}\sup_{t\in[0,T]}\sup_{x\in[-L,L]} \left( |u_{\theta'}^\epsilon(t,x,\omega) - t\Ham(\theta') - \theta' x| + |u_\theta^\epsilon(t,x,\omega) - u_{\theta'}^\epsilon(t,x,\omega)| \right)\\
		&\quad + T|\Ham(\theta) - \Ham(\theta')| + |\theta - \theta'|L\\
		&\le T\left[ (\kappa_\theta(\omega) + 1)\delta + |\Ham(\theta) - \Ham(\theta')| \right] + \left[ \delta + |\theta - \theta'| \right] L.
	\end{align*}
	Since $|\theta - \theta'|\le\frac{\delta}{2}$, $\Ham$ is continuous and $\delta\in(0,1)$ is arbitrary, we conclude that
	\[ \lim_{\epsilon\to0}\sup_{t\in[0,T]}\sup_{x\in[-L,L]} |u_\theta^\epsilon(t,x,\omega) - t\Ham(\theta) - \theta x| = 0.\qedhere \]
\end{proof}

\begin{proof}[Proof of Corollary \ref{cor:velinim}]
	The desired result follows readily from Theorem \ref{thm:homlin} and \cite[Theorem 3.1]{DK17}. The set $\Omega_0\in\mathcal{F}$ with $\mathbb{P}(\Omega_0) = 1$ is the one in Theorem \ref{thm:homlin}.
\end{proof}

\section*{Acknowledgments}

The author is grateful to E.\ Kosygina for helpful comments, including her concrete suggestions that simplified the statements and the proofs of Theorems \ref{thm:staticHJ}--\ref{thm:thetasOK}. The author also thanks A.\ Davini for his valuable feedback on a preliminary version of the manuscript.

\section*{Appendices}

\appendices

\section{On the scaled hill condition}\label{app:scaledhill}

With the notation
\[ s(x,\omega) = \int_0^x\frac{dy}{a(y,\omega)} \]
that we used in Subsection \ref{sub:flat}, the scaled hill condition \eqref{eq:scaledhill} reads as follows:
\begin{align*}
	&\text{for every $h\in(0,1)$, $C>0$ and $\mathbb{P}$-a.e.\ $\omega$, there is an interval $[L_1,L_2]$ such that}\\
	&s(L_2,\omega) - s(L_1,\omega) \ge C\ \text{and}\ V(\,\cdot\,,\omega) \ge h\ \text{on}\ [L_1,L_2].
\end{align*}
It is a refinement of the following condition:
\begin{equation}\label{eq:longhill}
	\mathbb{P}( V(\,\cdot\,,\omega) \ge h\ \text{on}\ [0,L] ) > 0\ \text{for every $h\in(0,1)$ and $L>0$.}
\end{equation}

\begin{proposition}\label{prop:long}
	Assume \eqref{eq:stat} and \eqref{eq:aVcon1}. 
	\begin{itemize}
		\item [(a)] 
		\eqref{eq:longhill} implies the scaled hill condition.
		\item [(b)] If $\mathbb{P}( a(0,\omega) \ge \kappa ) = 1$ for some $\kappa > 0$, then \eqref{eq:longhill} is equivalent to the scaled hill condition.
	\end{itemize}
\end{proposition}

\begin{proof}
	Suppose \eqref{eq:longhill} holds. Fix any $h\in(0,1)$ and $C>0$. By ergodicity, for $\mathbb{P}$-a.e.\ $\omega$, there is a $z = z(\omega)\in\mathbb{R}$ such that $V(\,\cdot\,,\omega) \ge h$ on $[z,z + C]$. Since $s(z + C,\omega) - s(z,\omega) \ge C$, we conclude that the scaled hill condition holds. This proves part (a).
	
	Suppose $\mathbb{P}( a(0,\omega) \ge \kappa ) = 1$ for some $\kappa > 0$ and the scaled hill condition holds. Fix any $h\in(0,1)$, $L>0$ and $C > L/{\kappa}$. For $\mathbb{P}$-a.e.\ $\omega$, there is an interval $[L_1,L_2]$ such that
	\[ \frac{L}{\kappa} < C \le s(L_2,\omega) - s(L_1,\omega) \le \frac{L_2 - L_1}{\kappa} \quad\text{and}\quad V(\,\cdot\,,\omega) \ge h\ \text{on}\ [L_1,L_2]. \]
	Therefore, \[ \mathbb{P}( V(\,\cdot\,,\omega) \ge h\ \text{on\ $[z,z+L]$ for some $z\in\mathbb{Q}$} ) = 1, \]
	and \eqref{eq:longhill} follows from stationarity (and the countability of $\mathbb{Q}$). This proves part (b).
\end{proof}

In general, \eqref{eq:longhill} is not equivalent to the scaled hill condition. In fact, the latter can hold while the former fails in a remarkable way. To illustrate this, we introduce yet another condition:
\begin{equation}\label{eq:singularhill}
\begin{aligned}
&\text{for every $c\in(0,1)$ and $\mathbb{P}$-a.e.\ $\omega$, there is a $z \in \mathbb{R}$ such that}\\
&a(z,\omega) \le c\ \text{and}\ V(z,\omega) \ge 1-c.
\end{aligned}
\end{equation}

\begin{proposition}
	If $a(\,\cdot\,,\omega)\in\Lip(\mathbb{R})$ and $V(\,\cdot\,,\omega)\in\UC(\mathbb{R})$ for every $\omega\in\Omega$, then \eqref{eq:singularhill} implies the scaled hill condition.
\end{proposition}

\begin{proof}
	Suppose \eqref{eq:singularhill} holds. Fix any $h\in(0,1)$ and $\epsilon\in(0,1-h)$. There is an $\Omega_0\in\mathcal{F}$ with $\mathbb{P}(\Omega_0) = 1$ such that, for every $\omega\in\Omega_0$ and $c\in(0,1 - h - \epsilon)$, there is a $z = z(\omega)\in\mathbb{R}$ such that
	$a(z,\omega) \le c$ and $V(z,\omega) \ge 1 - c > h + \epsilon$ (since it suffices to consider $c\in(0,1 -h - \epsilon)\cap\mathbb{Q}$). If $a(\,\cdot\,,\omega)\in\Lip(\mathbb{R})$ and $V(\,\cdot\,,\omega)\in\UC(\mathbb{R})$, then there exist $K = K(\omega) > 0$ and $\delta = \delta(\omega) > 0$ such that $a(x,\omega) \le c + K|x-z|$ for every $x\in\mathbb{R}$ and $V(\,\cdot\,,\omega) \ge h$ on $[z-\delta,z+\delta]$. Note that
	\begin{align*}
	s(z+\delta,\omega) - s(z-\delta,\omega) \ge \int_{z-\delta}^{z+\delta}\frac{dx}{c + K|x-z|} = 2\int_0^\delta\frac{dy}{c + Ky} = \frac{2}{K(\omega)}\log\left(\frac{c + K(\omega)\delta(\omega)}{c}\right).
	\end{align*}
	Since $c\in(0,1 - h - \epsilon)$ is arbitrary, we conclude that the scaled hill condition holds.
\end{proof}

The hill condition \eqref{eq:longhill} and an analogous valley condition (obtained by replacing $V$ with $1-V$) were initially formulated in \cite{YZ19} for potentials $V:\mathbb{Z}\times\Omega\to[0,1]$ in the context of a homogenization problem for controlled random walks. These hill and valley conditions were subsequently adapted in \cite{KYZ20} to our setting and utilized to prove that \eqref{eq:asilhuzur} homogenizes when $G$ is not quasiconvex but given by $G(p) = \frac1{2}\min\{(p-c)^2,(p+c)^2\}$ for some $c>0$. This is the continuous version of the main result in \cite{YZ19} and it has recently been generalized in \cite{DK20+} to the case where $G$ is the minimum of two or more convex functions with the same absolute minimum. Note that there was no need to introduce scaled hill and valley conditions in \cite{YZ19,KYZ20} because they assume that $a \equiv \frac1{2}$. In contrast, the latest version of \cite{DK20+} adopts such scaled conditions that originate from \eqref{eq:scaledhill} but are defined slightly differently to cover possibly degenerate diffusion coefficients.

In the discrete setting, with our other assumptions in place, the hill condition \eqref{eq:longhill} is satisfied when the law of $(V(x,\omega))_{x\in\mathbb{Z}}$ under $\mathbb{P}$ is a product measure, and more generally when the law of $(V(x,\omega))_{0 \le x \le L}$ under $\mathbb{P}$ is mutually absolutely continuous with the product measure formed by its marginals for every $L>0$ (see \cite[Example 1.2]{YZ19}). We can extend such potentials from $\mathbb{Z}$ to $\mathbb{R}$ by linear interpolation, make a change of variable that maps $\mathbb{Z}$ to a suitable stationary point process, perform a mollification if necessary, and thereby obtain stationary potentials that satisfy \eqref{eq:longhill} (and hence the scaled hill condition) as well as any desired mixing (including finite-range dependence) or regularity condition (see \cite[Example B.1]{DK20+}). Moreover, a variant of this construction yields potentials that satisfy \eqref{eq:longhill} but are not even weakly mixing (see \cite[Example 1.3]{YZ19}).

It is also easy to construct stationary potentials that satisfy \eqref{eq:longhill} without starting from the discrete setting, e.g., by taking moving averages of truncated increments of a two-sided Brownian motion or Poisson process (see \cite[Example 1.3]{KYZ20}) or by considering two-sided Brownian motion that is confined to $[0,1]$ under reflecting boundary conditions and then mollified appropriately (see \cite[Example B.3]{DK20+}). In fact, for any stationary potential $V:\mathbb{R}\times\Omega\to[0,1]$, the hill condition \eqref{eq:longhill} holds unless $x\mapsto V(x,\omega)$ is almost surely rigid in the sense that it cannot stay arbitrarily close to a given height for arbitrarily long. From the perspective of stationary \& ergodic processes, it can be argued that such rigid potentials are not typical (see \cite[Section B.3]{DK20+}).

The scaled hill condition fails most notably when $x\mapsto (a(x,\omega),V(x,\omega))$ is periodic (which is the prime example of rigidity in the above sense). However, in that case, homogenization follows from compactness arguments that prove the existence of a periodic (and hence bounded) corrector for every direction (see Subsection \ref{sub:literature} and the references therein).

\section{A variant of the Gr\"onwall-Bellman lemma}\label{app:gronwall}

\begin{lemma}\label{lem:gronwall}
	Given any $K>0$ and $L_1,L_2\in\mathbb{R}$ such that $L_1 < L_2$, suppose
	\begin{align*}
	&\text{$a:[L_1,L_2] \to (0,+\infty)$ is in $\Con([L_1,L_2])$,}\\
	&\text{$h:[L_1,L_2] \to [0,K]$ is in $\Con^1([L_1,L_2])$ and $0 < h(L_1) \le K$, and}\\
	&\text{$m:[0,K] \to [0,+\infty)$ is in $\Con([0,K])$, $m(0) = 0$ and $0 < m(q) \le q$ for every $q\in(0,K]$.}
	\end{align*}
	If \[ a(x)h'(x) + m(h(x)) \le 0 \] for every $x\in(L_1,L_2)$, then
	\begin{equation}\label{eq:tostcu}
	h(x) \le \Phi^{-1}\left(\int_{L_1}^x\frac{dy}{a(y)}\right)
	\end{equation}
	for every $x\in(L_1,L_2)$, where $\Phi:(0,K] \to (0,+\infty)$ is defined by
	\[ \Phi(p) = \int_p^K\frac{dq}{m(q)} \]
	and its inverse satisfies
	\begin{equation}\label{eq:dudukiki}
	\lim_{z\to+\infty} \Phi^{-1}(z) = 0.
	\end{equation}
\end{lemma}

\begin{proof}
	For every $p \in (0,K]$, let
	\[ \Psi(p) = \int_p^{h(L_1)}\frac{dq}{m(q)}. \]
	Note that $\Psi(p) \le \Phi(p)$ because $h(L_1) \le K$. By the chain rule,
	\[ \frac{d}{dx}\Psi(h(x)) = -\frac{h'(x)}{m(h(x))} \ge \frac1{a(x)}. \]
	Integrating both sides and using $\Psi(h(L_1)) = 0$, we get
	\[ \Phi(h(x)) \ge \Psi(h(x)) \ge \int_{L_1}^x\frac{dy}{a(y)}. \]
	Since $\Phi$ is strictly decreasing, \eqref{eq:tostcu} holds. Finally, \eqref{eq:dudukiki} follows from
	the observation that
	\[ \lim_{p\downarrow0}\Phi(p) = \int_0^K\frac{dq}{m(q)} \ge \int_0^K\frac{dq}{q} = +\infty. \qedhere \]	
\end{proof}

\bibliographystyle{abbrv}
\bibliography{1d_1well_viscous_ref}

\begin{thebibliography}{10}

\bibitem{AC18}
S.~N. Armstrong and P.~Cardaliaguet.
\newblock Stochastic homogenization of quasilinear {H}amilton-{J}acobi
  equations and geometric motions.
\newblock {\em J. Eur. Math. Soc. (JEMS)}, 20(4):797--864, 2018.

\bibitem{AS13}
S.~N. Armstrong and P.~E. Souganidis.
\newblock Stochastic homogenization of level-set convex {H}amilton-{J}acobi
  equations.
\newblock {\em Int. Math. Res. Not. IMRN}, (15):3420--3449, 2013.

\bibitem{AT14}
S.~N. Armstrong and H.~V. Tran.
\newblock Stochastic homogenization of viscous {H}amilton-{J}acobi equations
  and applications.
\newblock {\em Anal. PDE}, 7(8):1969--2007, 2014.

\bibitem{ATY15}
S.~N. Armstrong, H.~V. Tran, and Y.~Yu.
\newblock Stochastic homogenization of a nonconvex {H}amilton-{J}acobi
  equation.
\newblock {\em Calc. Var. Partial Differential Equations}, 54(2):1507--1524,
  2015.

\bibitem{ATY16}
S.~N. Armstrong, H.~V. Tran, and Y.~Yu.
\newblock Stochastic homogenization of nonconvex {H}amilton-{J}acobi equations
  in one space dimension.
\newblock {\em J. Differential Equations}, 261(5):2702--2737, 2016.

\bibitem{CS17}
P.~Cardaliaguet and P.~E. Souganidis.
\newblock On the existence of correctors for the stochastic homogenization of
  viscous {H}amilton-{J}acobi equations.
\newblock {\em C. R. Math. Acad. Sci. Paris}, 355(7):786--794, 2017.

\bibitem{CIL92}
M.~G. Crandall, H.~Ishii, and P.-L. Lions.
\newblock User's guide to viscosity solutions of second order partial
  differential equations.
\newblock {\em Bull. Amer. Math. Soc. (N.S.)}, 27(1):1--67, 1992.

\bibitem{D19}
A.~Davini.
\newblock Existence and uniqueness of solutions to parabolic equations with
  superlinear {H}amiltonians.
\newblock {\em Commun. Contemp. Math.}, 21(1):1750098, 25, 2019.

\bibitem{DK17}
A.~Davini and E.~Kosygina.
\newblock Homogenization of viscous and non-viscous {HJ} equations: a remark
  and an application.
\newblock {\em Calc. Var. Partial Differential Equations}, 56(4):Paper No. 95,
  21, 2017.

\bibitem{DK20+}
A.~Davini and E.~Kosygina.
\newblock Stochastic homogenization of a class of nonconvex viscous {HJ}
  equations in one space dimension, 2020.
\newblock Eprint arXiv:math.AP/2002.02263.

\bibitem{DS09}
A.~Davini and A.~Siconolfi.
\newblock Exact and approximate correctors for stochastic {H}amiltonians: the
  1-dimensional case.
\newblock {\em Math. Ann.}, 345(4):749--782, 2009.

\bibitem{E89}
L.~C. Evans.
\newblock The perturbed test function method for viscosity solutions of
  nonlinear {PDE}.
\newblock {\em Proc. Roy. Soc. Edinburgh Sect. A}, 111(3-4):359--375, 1989.

\bibitem{E92}
L.~C. Evans.
\newblock Periodic homogenisation of certain fully nonlinear partial
  differential equations.
\newblock {\em Proc. Roy. Soc. Edinburgh Sect. A}, 120(3-4):245--265, 1992.

\bibitem{FeFZ19+}
W.~M. Feldman, J.-B. Fermanian, and B.~Ziliotto.
\newblock An example of failure of stochastic homogenization for viscous
  {H}amilton-{J}acobi equations without convexity, 2019.
\newblock Eprint arXiv:math.AP/1905.07295.

\bibitem{FeS17}
W.~M. Feldman and P.~E. Souganidis.
\newblock Homogenization and non-homogenization of certain non-convex
  {H}amilton-{J}acobi equations.
\newblock {\em J. Math. Pures Appl. (9)}, 108(5):751--782, 2017.

\bibitem{FleSon06}
W.~H. Fleming and H.~M. Soner.
\newblock {\em Controlled {M}arkov processes and viscosity solutions},
  volume~25 of {\em Stochastic Modelling and Applied Probability}.
\newblock Springer, New York, second edition, 2006.

\bibitem{G16}
H.~Gao.
\newblock Random homogenization of coercive {H}amilton-{J}acobi equations in
  1d.
\newblock {\em Calc. Var. Partial Differential Equations}, 55(2):Paper No. 30,
  39, 2016.

\bibitem{G19}
H.~Gao.
\newblock Stochastic homogenization of certain nonconvex {H}amilton-{J}acobi
  equations.
\newblock {\em J. Differential Equations}, 267(5):2918--2949, 2019.

\bibitem{K07}
E.~Kosygina.
\newblock Homogenization of stochastic {H}amilton-{J}acobi equations: brief
  review of methods and applications.
\newblock In {\em Stochastic analysis and partial differential equations},
  volume 429 of {\em Contemp. Math.}, pages 189--204. Amer. Math. Soc.,
  Providence, RI, 2007.

\bibitem{KRV06}
E.~Kosygina, F.~Rezakhanlou, and S.~R.~S. Varadhan.
\newblock Stochastic homogenization of {H}amilton-{J}acobi-{B}ellman equations.
\newblock {\em Comm. Pure Appl. Math.}, 59(10):1489--1521, 2006.

\bibitem{KYZ20}
E.~Kosygina, A.~Yilmaz, and O.~Zeitouni.
\newblock Homogenization of a class of one-dimensional nonconvex viscous
  {H}amilton-{J}acobi equations with random potential.
\newblock {\em Comm. Partial Differential Equations}, 45(1):32--56, 2020.

\bibitem{LPV87}
P.-L. Lions, G.~Papanicolaou, and S.~R.~S. Varadhan.
\newblock Homogenization of {H}amilton-{J}acobi equations.
\newblock Unpublished manuscript (available at
  www.researchgate.net/publication/246383838), 1987.

\bibitem{LS03}
P.-L. Lions and P.~E. Souganidis.
\newblock Correctors for the homogenization of {H}amilton-{J}acobi equations in
  the stationary ergodic setting.
\newblock {\em Comm. Pure Appl. Math.}, 56(10):1501--1524, 2003.

\bibitem{LS05}
P.-L. Lions and P.~E. Souganidis.
\newblock Homogenization of ``viscous'' {H}amilton-{J}acobi equations in
  stationary ergodic media.
\newblock {\em Comm. Partial Differential Equations}, 30(1-3):335--375, 2005.

\bibitem{QTY18}
J.~Qian, H.~V. Tran, and Y.~Yu.
\newblock Min--max formulas and other properties of certain classes of
  nonconvex effective {H}amiltonians.
\newblock {\em Math. Ann.}, 372(1-2):91--123, 2018.

\bibitem{RT00}
F.~Rezakhanlou and J.~E. Tarver.
\newblock Homogenization for stochastic {H}amilton-{J}acobi equations.
\newblock {\em Arch. Ration. Mech. Anal.}, 151(4):277--309, 2000.

\bibitem{S99}
P.~E. Souganidis.
\newblock Stochastic homogenization of {H}amilton-{J}acobi equations and some
  applications.
\newblock {\em Asymptot. Anal.}, 20(1):1--11, 1999.

\bibitem{Y20+}
A.~Yilmaz.
\newblock Stochastic homogenization and effective {H}amiltonians of {HJ}
  equations in one space dimension: the double-well case, 2020.
\newblock Eprint arXiv:math.AP/2007.07854.

\bibitem{YZ19}
A.~Yilmaz and O.~Zeitouni.
\newblock Nonconvex homogenization for one-dimensional controlled random walks
  in random potential.
\newblock {\em Ann. Appl. Probab.}, 29(1):36--88, 2019.

\bibitem{Z17}
B.~Ziliotto.
\newblock Stochastic homogenization of nonconvex {H}amilton-{J}acobi equations:
  a counterexample.
\newblock {\em Comm. Pure Appl. Math.}, 70(9):1798--1809, 2017.

\end{thebibliography}

\end{document}